\documentclass[reqno]{amsart}
\usepackage[english]{babel}

\usepackage[T1]{fontenc}
\usepackage[utf8]{inputenc}
\usepackage[margin=1in]{geometry}
\usepackage{amsmath,amssymb,amsthm,amsfonts,mathtools}
\usepackage[hidelinks]{hyperref}
\usepackage{enumitem}
\usepackage{mathrsfs}
\usepackage{tikz-cd}
\usepackage{stmaryrd}
\usepackage{microtype}
\usepackage{csquotes}
\usepackage{comment}

\newtheorem{theorem}{Theorem}[section]
\newtheorem{proposition}[theorem]{Proposition}
\newtheorem{lemma}[theorem]{Lemma}
\newtheorem{corollary}[theorem]{Corollary}
\theoremstyle{definition}
\newtheorem{definition}[theorem]{Definition}
\newtheorem{remark}[theorem]{Remark}

\newtheorem{claim}[theorem]{Claim}
\newtheorem{ltheorem}{Theorem}

\newcommand{\weak}{\overset{*}{\rightharpoonup} }

\newcommand{\emp}{\mathsf e}
\newcommand{\Leb}{\operatorname{Leb}}
\newcommand{\Prob}{\mathcal{P}}

\newcommand{\acc}{\operatorname{acc}}

\newcommand{\supp}{\operatorname{supp}}

\newcommand{\N}{\mathbb{N}}
\newcommand{\Sone}{\mathbb{S}^1}
\newcommand{\eps}{\varepsilon}
\newcommand{\wto}{\overset{*}{\rightharpoonup}}
\newcommand{\Wone}{W_1}
\newcommand{\cM}{\mathcal{M}}
\newcommand{\cG}{\mathcal{G}}

\title{Emergence of partially hyperbolic skew-products with circle fibers}

\author[O. Costa]{Odylo Costa}
\address{Departamento de Matem\'atica, Universidade Federal do Cear\'a (UFC), Campus do Pici,
Bloco 914, CEP 60455-760. Fortaleza -- CE, Brasil}
\email{odylo.costa@mat.ufc.br}

\author[M. Poletti]{Mauricio Poletti}
\address{Departamento de Matem\'atica, Universidade Federal do Cear\'a (UFC), Campus do Pici,
Bloco 914, CEP 60455-760. Fortaleza -- CE, Brasil}
\email{mpoletti@mat.ufc.br}

\begin{document}

\begin{abstract}
    We study the statistical behavior of partially hyperbolic skew-products over Axiom A diffeomorphisms using emergence. Generically, these systems satisfy the Palis conjecture: they have a finite number of physical measures whose basins together have full volume. In particular, for Lebesgue almost every point, the empirical measures converge to one of finitely many physical measures. We study the non-generic case and prove that, even when the sequence of empirical measures may have uncountably many accumulation points, these systems have low emergence.
\end{abstract}

    \maketitle
    
\section{Introduction}
   Chaotic dynamical systems exhibit complex and unpredictable behavior. Despite being governed by deterministic laws, their sensitive dependence on initial conditions limits the long-term prediction of individual trajectories. For this reason, we study their statistical behavior.

By statistical behavior, we mean the behavior of averages of observables $\varphi: M\to\mathbb{R}$ along the orbits of a dynamical system $F: M\to M$. More specifically, given a point $x\in M$, we consider
$
\frac{1}{n}\sum_{j=0}^{n-1}\varphi(F^j(x))
$
and seek to understand its behavior as $n$ tends to infinity.

To study these averages, we use empirical measures. Given $x\in M$, define
\[
\emp_n(x):=\frac{1}{n}\sum_{j=0}^{n-1}\delta_{F^j(x)},
\]
where $\delta_p$ is the Dirac measure concentrated at $p\in M$. The measures $\emp_n(x)$ are called \emph{empirical measures}, and we study their accumulation points in the weak-$*$ topology. Observe that, for every continuous observable $\varphi:M\to\mathbb{R}$,
$
\frac{1}{n}\sum_{j=0}^{n-1}\varphi(F^j(x))
=\int\varphi\,d\emp_n(x).
$
Thus, understanding the limiting behavior of these averages for all continuous observables is equivalent to understanding the accumulation points of $\emp_n(x)$.

We are interested in the empirical measures of Lebesgue almost every point. Given an invariant probability measure $\mu$, we define its \emph{basin} by
$
B(\mu):=\{x\in M:\emp_n(x)\weak\mu\},
$
where $\weak$ denotes convergence in the weak-$*$ topology. A measure $\mu$ is called \emph{physical} if $B(\mu)$ has positive Lebesgue measure.

Palis conjectured \cite{Pa00} that generic diffeomorphisms should have finitely many physical measures whose basins together have full Lebesgue measure. This holds for hyperbolic diffeomorphisms satisfying Axiom A; see \cite{Si72,Rue76,Bowen1975}.

   The existence of finitely many physical measures whose basins together have full Lebesgue measure has also been established for several classes of partially hyperbolic diffeomorphisms, including systems with mostly contracting center \cite{BV00}, systems satisfying suitable nonuniform expansion conditions along the center-unstable direction \cite{ABV00}, and systems with mostly expanding center in the sense of Gibbs $u$-states \cite{AV18}.
   
   We will focus on a particular class of partially hyperbolic diffeomorphisms, namely partially hyperbolic skew-products with circle fibers over Axiom A diffeomorphisms; see section~\ref{sec.def} for the definition. It was proved in \cite{VianaYang2013} that this class satisfies the Palis conjecture: $C^r$ generically, it has a finite number of physical measures whose basins together have full Lebesgue measure. In this work, we want to understand the non-generic case. 
   
   \emph{How big can the set of accumulation points of empirical measures of Lebesgue almost every point be?} To measure how complex this set can be, we use the concept of emergence introduced by Berger~\cite{BergerSteklov2017}. The emergence of a system $F$ is a function $\mathcal{E}_{\Leb}(F)$ which associates, to each positive real number $\eps>0$, the number of probability measures needed to approximate the asymptotic behavior of $\emp_n(x)$ up to (average) error $\eps$. A precise definition of $\mathcal{E}_{\Leb}(F)$ is given in Section~\ref{sec.def}.

   Observe that when $\eps\to 0^+$, different things can happen for different systems. If $F$ is an Axiom A diffeomorphism, for example, then $\mathcal{E}_{\Leb}(F)$ is uniformly bounded by the number of attractors for every $\eps>0$; if $F=\mathrm{Id}$ in $\Sone$ then, as $\eps\to 0^+$, this number grows to $+\infty$ as $\eps^{-1}$. In the last decade, some systems with emergence of superpolynomial growth were discovered (see, for example, \cite{BergerBochiAIM2021},\cite{berger2026analytic},\cite{delaporte2025analytic}). We say that $F$ has \emph{low emergence} if its emergence has at most polynomial growth, i.e., if there exists $\alpha\geq0$ such that $\mathcal E_{\Leb}(F)(\eps)=O(\eps^{-\alpha})$.
   
    Recently, Costa and Santiago~\cite{CostaSantiago2026} showed that, in the absence of critical points, smooth one-dimensional systems have at most polynomial emergence. Thus, one could ask whether partially hyperbolic systems with one-dimensional center dynamics have low emergence. In this paper, we show that for partially hyperbolic skew-products, the statistical behavior is, in some sense, tame. Our main result is the following.
    
    \begin{ltheorem}\label{thm:low-emergence}
        Every $C^{1+}$ partially hyperbolic skew-product with circle fibers over an Axiom A diffeomorphism has low emergence.
    \end{ltheorem}

   To prove this result, we give a detailed description of the $u$-Gibbs measures of $F$, whose conditional measures along strong unstable manifolds are absolutely continuous with respect to leafwise Lebesgue measure. This analysis also yields results of independent interest, which we present in the following section.
    
We now introduce the main definitions and state the more technical results.

\subsection*{AI Declaration} 
We used ChatGPT (Open AI) to identify relevant references, improve grammar, and assist with checking proofs. All mathematical work was carried out by the authors.

\subsection*{Acknowledgment} 
O.C. was supported by Instituto Serrapilheira grant “Jangada Dinâmica: Impulsionando Sistemas Dinâmicos na Região Nordeste”
M.P. was supported by CAPES-Finance Code 001, Instituto Serrapilheira, grant number Serra-R-2211-41879 and FUNCAP, grant AJC 06/2022.

%----------------------------------------------------------------------
%----------------------------------------------------------------------
\section{Definitions and Statements}\label{sec.def}
%----------------------------------------------------------------------
%----------------------------------------------------------------------

    Let $M$ be a compact metric space, $F$ a continuous dynamical system on $M$, and $\Prob(M)$ be the space of probability measures on $M$, equipped with the $1$-Wasserstein distance $\Wone$. This metric induces the weak-$*$ topology and is defined by:
	$$
	\Wone(\mu,\nu)=\inf_{\pi \in \Pi(\mu,\nu)}\int_{M\times M} d(x,y) \ d\pi(x,y),
	$$
	where $d$ is the metric on $M$, $\Pi(\mu,\nu)\coloneqq \left\{\pi \in \mathcal{P}(M\times M) \mid (p_{1})_{\ast}\pi=\mu \text{ and }(p_{2})_{\ast}\pi=\nu\right\}$, and $p_{i}\colon M\times M \to M$ are the projections onto the first and second coordinates. The elements of $\Pi(\mu,\nu)$ are called couplings or transport plans between $\mu$ and $\nu$. See \cite{Villani2003Topics} for equivalent definitions.

    For a Borel probability measure $\eta$ on $M$, not necessarily invariant, define the emergence function of $F$ with respect to $\eta$ by
	\[
	\mathcal E_\eta(F)(\eps)
	:=
	\min\left\{
	N\in\N:
	\begin{array}{l}
	\exists\nu_1,\dots,\nu_N\in\Prob(M)\text{ such that}\\[1mm]
	\displaystyle
	\limsup_{n\to\infty}
	\int_M
	\min_{1\le i\le N}
	\Wone(\emp_n(z),\nu_i)
	\,d\eta(z)
	\le\eps
	\end{array}
	\right\}.
	\]

	Following \cite{BergerSteklov2017,BergerBochiAIM2021}, we say that a system has \emph{low emergence} if its emergence grows at most polynomially in $\eps^{-1}$, i.e., if
	\[
	\limsup_{\eps\to0^+}
	\frac{\log\mathcal E_\eta(F)(\eps)}
	{-\log\eps}<+\infty.
	\]

    From now on we will take $\eta$ to be the Lebesgue measure. 

    Emergence gives a way to classify systems by the complexity of their statistical behavior. Systems such as ergodic systems or Axiom A diffeomorphisms have uniformly bounded emergence, so $\mathcal{E}_{\Leb}(f)(\eps)=O(1)$. The identity and some systems with infinitely many sinks can have an unbounded emergence function, but it is still low in the above sense: for example, there exists $\alpha>0$ such that $\mathcal{E}_{\Leb}(f)(\eps)=O(\eps^{-\alpha})$ when $\eps\to 0$. If a system has emergence with superpolynomial growth, we say it has high emergence.
    
    As explained in the introduction, Theorem~\ref{thm:low-emergence} is about showing low emergence for systems with some form of hyperbolicity, namely the partially hyperbolic skew-products.

   Let $N$ be a smooth compact manifold and $f\colon N\to N$ be an Axiom A diffeomorphism. This means $\Omega(f)=\overline{\operatorname{Per}(f)}$ is a hyperbolic set. This means that there exist constants $C>0$ and $0<\lambda<1$ and a $Df$-invariant splitting for every $x\in \Omega(f)$ 
   $$
   T_x N=E^s_f(x)\oplus E^u_f(x),\quad \text{s.t. }  \|Df^n(x)\mid E^s_f\|\leq C\lambda^n\text{ and } \|Df^{-n}(x)\mid E^u_f\|\leq C\lambda^n\text{ for }n\geq 0.
   $$
   
Define the skew-product diffeomorphism 
	\[
	F\colon N\times\Sone
	\longrightarrow
	N\times\Sone,
	\qquad
	F(x,\theta)=(f(x),f_x(\theta)),
	\]
	where every $f_x$ is an orientation-preserving circle diffeomorphism. Also let $\pi:N\times \Sone\to N$, to be the projection $\pi(x,\theta)=x$. We endow $N\times\Sone$ with the product Riemannian metric and denote by $\Leb$ its normalized Riemannian volume.
    
	Assume that $F$ is $C^{1+\alpha}$. For $n\geq1$ write $f^n_{x}:=f_{f^{n-1}(x)}\circ \cdots \circ f_{x}$. Assume that there exists an integer $n\geq 1$ such that for every $x\in \Omega(f)$, $\theta\in \Sone$, we have
    $$ 2\|Df^n|_{E^s_{f}(x)}\| < |(f^n_x)'(\theta)|<\frac12 m(Df^n|_{E^u_{f}(x)}),$$ where $m(A)=\inf_{\|v\|=1}\|Av\|$ denotes the conorm of $A$. If one of the extremal bundles is trivial, the corresponding inequality is omitted in the definition.
    
    This implies that $F$ is partially hyperbolic over $\Omega(f)\times \Sone$ with splitting
	\[
	T_z(N\times\Sone)=E^s_z\oplus E^c_z\oplus E^u_z,
	\qquad
	\dim E^c=1,
	\]
	and that $E^c$ is tangent to the foliation $\{\pi(z)\}\times S^1$, $z\in \Omega(f)\times \Sone$.

    By the stable manifold theorem, there exist, for every $z\in \Omega(f)\times \Sone $, strong stable and unstable manifolds $W^s_F(z)$ and $W^u_F(z)$, respectively. 
    
    An invariant measure $\mu$ is called \emph{$u$-Gibbs} if its disintegration along the unstable manifold is absolutely continuous with respect to the volume measure along the corresponding unstable manifold.

   Recall that, for Lebesgue-a.e. point in the basin of each partially hyperbolic attracting set $\Lambda\times\Sone$, every weak-$*$ accumulation point of the sequence of empirical measures is a $u$-Gibbs measure; see Remark~\ref{rmk:acc-u-gibbs}. Recall also that, given an ergodic invariant measure $\mu$, the \emph{center Lyapunov exponent} is defined as $\lambda^c(\mu)=\int \log |f'_x(\theta)|d\mu(x,\theta)$.

    We have the following result
    \begin{ltheorem}\label{thm.empirical}
        For Lebesgue almost every $z\in N\times \Sone$ either
        \begin{itemize}
            \item $\emp_n(z)\weak \mu$, where $\mu$ is an ergodic
        $u$-Gibbs measure with $\lambda^c(\mu)<0$, or
            \item there exist ergodic $u$-Gibbs measures $\mu^-,\mu^+$ with   $\lambda^c(\mu^-),\lambda^c(\mu^+)\geq0$, such that every accumulation point of $\emp_n(z)$ is a convex combination of $\mu^-$ and $\mu^+$. 
            We allow $\mu^-=\mu^+$.
        \end{itemize}
    \end{ltheorem}

    Viana and Yang \cite{VianaYang2013} proved that, generically, partially hyperbolic skew-products with circle fibers have a finite number of ergodic $u$-Gibbs measures, each with negative center exponent. Observe that in the non-generic case the situation is very different. For example, a trivial skew-product with $f_x$ being the identity on $\Sone$ has an uncountable number of $u$-Gibbs measures, all of them with zero center exponent. In this case, it is easy to prove that $\mathcal{E}_{Leb}(F)(\varepsilon)= O(\varepsilon^{-1})$. We will prove that actually this is the greatest possible rate of growth of emergence.

    \begin{ltheorem}\label{thm:emergencia-ordem-pol}
    The emergence of the map $F$ satisfies $\mathcal{E}_{\Leb}(F)(\eps)=O(\eps^{-1})$.
\end{ltheorem}
    
    Theorem~\ref{thm:low-emergence} is a direct consequence of this theorem.

%----------------------------------------------------------------------
%----------------------------------------------------------------------    
\section{Invariant measures of skew-products}
%----------------------------------------------------------------------
%----------------------------------------------------------------------
In this section we recall results of Furstenberg \cite{Furstenberg1961} on invariant measures of skew-products with circle fibers that project to a fixed ergodic measure.

	Let $\Omega$ be a compact metric space, let $g: \Omega \to \Omega$ be a homeomorphism, and let $m$ be a $g$-invariant ergodic Borel probability measure on $\Omega$. Let $G:\Omega\times \Sone\to \Omega\times \Sone$ be a skew-product homeomorphism given by $(\omega,t)\mapsto (g(\omega),g_\omega(t))$, with $g_\omega:\Sone\to \Sone$ order preserving homeomorphism for every $\omega\in \Omega$. For $n\geq1$, we write $g_\omega^n:=g_{g^{n-1}(\omega)}\circ\cdots\circ g_\omega$. We endow $\Omega \times \Sone$ with the sum product metric, where the distance in $\Sone$ is the geodesic one.

     We denote by
     $$
    \cM_m(G):=\left\{ \nu\in\Prob(\Omega\times\Sone):G_*\nu=\nu,\quad(\pi_1)_*\nu=m\right\}$$ the set of $G$-invariant probability measures projecting to $m$, and by $\cM_m^{\mathrm{erg}}(G)$ the subset of its ergodic elements.

%----------------------------------------------------------------------
%----------------------------------------------------------------------
    \subsection*{Invariant \texorpdfstring{$(p,q)$}{p,q}-graphs}
%----------------------------------------------------------------------
%----------------------------------------------------------------------

	Following \cite{JagerKeller2006}, we define the invariant $(p,q)$-graphs. For any $k\in\mathbb N$, define $(\Sone)^k_*= \{(x_1,\dots,x_k)\in (\Sone)^k:x_i\neq x_j,\ \forall\,1\leq i<j\leq k\}$. Let $p,q\in\N$ and 
	\[
	\varphi\colon\Omega\longrightarrow(\Sone)^{pq}_*,
	\qquad
	\varphi(\omega)
	=
	\bigl(\varphi_j^i(\omega)\bigr)_{
		\substack{1\le i\le p\\1\le j\le q}}
	\]
	be a measurable map defined $ m$-almost everywhere. We denote by $\Phi
	=
	\operatorname{graph}(\varphi)$ the corresponding subset of $\Omega\times\Sone$ defined by $\cup_{\omega\in \Omega}\{\omega\}\times \Phi_\omega$ where 
	\[
	\Phi_\omega
	=
	\left\{
	\varphi_j^i(\omega):
	1\le i\le p,\ 1\le j\le q
	\right\}.
	\]

	For each $i=1,\dots,p$, we also write
	\[
	\varphi^i
	:=
	(\varphi_1^i,\dots,\varphi_q^i)
	\]
	and denote by
	\[
	\Phi^i_\omega
	:=
	\{
	\varphi_1^i(\omega),\dots,\varphi_q^i(\omega)
	\}
	\]
	the corresponding $q$-point subset of the fiber over $\omega$.

			\begin{definition}\label{def:pq}
		The map $\varphi$ is called an \emph{invariant $(p,q)$-graph} if
		\[
		g_\omega(\Phi_\omega)
		=
		\Phi_{g\omega}
		\]
		for $m$-a.e.\ $\omega$, and the following two conditions hold:
		\begin{enumerate}[label=(MIG\arabic*), leftmargin=3.7em]
			\item $\Phi$ cannot be measurably decomposed into smaller
			disjoint invariant subgraphs;

			\item each $\varphi^i$ is $p$-periodic, in the sense that
			for $m$-a.e.\ $\omega$,
			$$ 
			g_\omega^p(\Phi^i_\omega)
			=
			\Phi^i_{g^p\omega},\quad
			\text{and }
			\quad
			g_\omega^r(\Phi^i_\omega)
			\neq
			\Phi^i_{g^r\omega},
			\quad
			r=1,\dots,p-1,
			$$ 
			also $\varphi^i$ cannot be further decomposed into smaller
			invariant or periodic subgraphs.
		\end{enumerate}
	\end{definition}

		In this notation, for $m$-a.e.\ $\omega$,
	$ 
	\Phi_\omega
	=
	\Phi^1_\omega\cup\cdots\cup\Phi^p_\omega
	$ 
	
    Since $\varphi$ is a $pq$-valued graph, these $pq$ points are
	pairwise distinct for $m$-a.e.\ $\omega$.  In particular, the
	$p$ subsets $ 
	\Phi^1_\omega,\dots,\Phi^p_\omega
	$ 
	are pairwise disjoint and each contains exactly $q$ points.

\begin{remark}
    The two integers $p$ and $q$ have different meanings. Each
$q$-valued graph $\varphi^i=(\varphi_1^i,\dots,\varphi_q^i)$ contains $q$ points on each fiber, while $p$ is its minimal period.
The $p$ graphs $\varphi^1,\dots,\varphi^p$ are cyclically permuted by the dynamics and together form the full
invariant $(p,q)$-graph.  Consequently, the latter contains exactly $pq$ points on almost every fiber.
\end{remark}

Before moving on, we observe that two invariant $(p,q)$-graphs must coincide or be disjoint $m$-a.e.:

\begin{lemma}\label{lem:irreducible-graphs}
	Let $\varphi$ and $\psi$ be two invariant $(p,q)$-graphs, with
	associated fiber sets $\Phi_\omega$ and $\Psi_\omega$. Then either $\Phi_\omega=\Psi_\omega$ for $m$-a.e.\ $\omega$, or $\Phi_\omega\cap\Psi_\omega=\emptyset$ for $m$-a.e.\ $\omega$.
\end{lemma}

\begin{proof}
	For $m$-a.e.\ $\omega$, define the measurable function $\omega \mapsto r(\omega):=\#\bigl(\Phi_\omega\cap\Psi_\omega\bigr)$. Since both graphs are invariant and every fiber map $g_\omega$ is injective,
	\[
	g_\omega
	\bigl(
	\Phi_\omega\cap\Psi_\omega
	\bigr)
	=
	\Phi_{g\omega}\cap\Psi_{g\omega}.
	\]
	Hence $r(g\omega)=r(\omega)$ for $m$-a.e.\ $\omega$ and, since $m$ is ergodic, there exists an integer $r_0\geq0$ such that $r(\omega)=r_0$ for $m$-a.e.\ $\omega$.

	If $r_0=0$, then $\Phi_\omega\cap\Psi_\omega=\emptyset$ for $m$-a.e.\ $\omega$. If $r_0>0$, then
	\[
	\omega\longmapsto
	\Phi_\omega\cap\Psi_\omega
	\]
	defines a nonempty finite-valued invariant subgraph of both
	$\Phi$ and $\Psi$. Since, by (MIG1), neither $\Phi$ nor $\Psi$
	contains a proper nonempty invariant subgraph, we must have
	\[
	\Phi_\omega\cap\Psi_\omega
	=
	\Phi_\omega
	=
	\Psi_\omega
	\]
	for $m$-a.e.\ $\omega$.
\end{proof}

\subsection*{The measure associated to
		invariant \texorpdfstring{$(p,q)$}{p,q}-graphs}
	
	Let $\varphi$ be an invariant $(p,q)$-graph and set $N:=pq$. We associate to $\varphi$ the probability measure
	\[
	\mu_\varphi
	:=
	\int_\Omega
	(\mu_\varphi)_\omega\,dm(\omega),
	\]
	where
	\[
	(\mu_\varphi)_\omega
	=
	\frac1{pq}
	\sum_{i=1}^p\sum_{j=1}^q
	\delta_{\varphi_j^i(\omega)}.
	\]
	Equivalently, for every measurable set
	$A\subset\Omega\times\Sone$,
	\[
	\mu_\varphi(A)
	=
	\frac1{pq}
	\sum_{i=1}^p\sum_{j=1}^q
	m
	\left(
	\left\{
	\omega:
	(\omega,\varphi_j^i(\omega))\in A
	\right\}
	\right).
	\]
	
	\begin{proposition}\label{prop:pq-measure}
		The measure $\mu_\varphi$ belongs to
		$\cM_m^{\mathrm{erg}}(G)$.
	\end{proposition}
	
	\begin{proof}
		The equality $(\pi_1)_*\mu_\varphi=m$ follows immediately from the definition. Since $g_\omega$ bijectively maps the $pq$ points of $\Phi_\omega$ onto the
		$pq$ points of $\Phi_{g\omega}$,
		\[
		(g_\omega)_*(\mu_\varphi)_\omega
		=
		(\mu_\varphi)_{g\omega}.
		\]
		Hence $\mu_\varphi$ is $G$-invariant.
		
		To prove ergodicity, let
		$A\subset\Omega\times\Sone$ be invariant modulo
		$\mu_\varphi$ and define
		\[
		r(\omega)
		:=
		\#
		\bigl(
		A_\omega\cap\Phi_\omega
		\bigr).
		\]
		Since $G$ maps $\Phi$ bijectively onto itself, $A$ is
		invariant modulo $\mu_\varphi$, and each atom of $(\mu_\varphi)_{\omega}$ has mass $1/pq$, the function $r$ is
		$g$-invariant for $m$-a.e.\ $\omega$. As $m$ is ergodic,
		$r$ is constant a.e. If $0<r<pq$, the sets $A_\omega\cap\Phi_\omega$ define a proper nonempty invariant subgraph of $\Phi$, contradicting (MIG1). Therefore, $r=0$ or $r=pq$ for $m$-a.e.\ $\omega$. Consequently,
		\[
		\mu_\varphi(A)
		=
		\int_\Omega
		\frac{r(\omega)}{pq}\,dm(\omega)
		\in\{0,1\}.
		\]
		Thus $\mu_\varphi$ is ergodic.
	\end{proof}

Given $s,t\in \Sone$, we denote by $(s,t)$ the open interval with endpoints $s$ and $t$, taken in the clockwise direction. Now let us state Furstenberg's Theorem \cite[Theorem~4.1]{Furstenberg1961}. Furstenberg uses different terminology; we state it as it is in \cite[Theorem~2.4]{JagerSkewProductNotes}; see also \cite[Theorem~2.5]{JagerKeller2006} and \cite[Theorem~2.10]{JagerStark2006} for the quasiperiodic setting.

	\begin{theorem}[Furstenberg]\label{thm:classification}
		Let $G$ be a skew-product such that $g$ is a homeomorphism and that $g_\omega$ is a homeomorphism for every $\omega\in \Omega$. Let $m$ be a $g$-invariant ergodic Borel probability measure.
		\begin{enumerate}[label=(\roman*), leftmargin=2.3em]
			\item If there exists no invariant graph, then $\cM_m(G)$ consists of a
			single measure $\mu$, and $\mu_\omega$ is nonatomic for $m$-a.e.\
			$\omega$.
            \item If there exists an invariant $(p,q)$-graph $\varphi$, then every
			$\mu\in\cM_m^{\mathrm{erg}}(G)$ is of the form
			\[
			\mu=\mu_\psi
			\]
			for an invariant $(p,q)$-graph $\psi$ with the \emph{same integers
			$p$ and $q$}. Moreover, for $m$ almost every $\omega$, if $(\varphi^i_j(\omega),\varphi^{i'}_{j'}(\omega))\cap \Phi(\omega)=\emptyset$, then $[\varphi^i_j(\omega),\varphi^{i'}_{j'}(\omega))$ contains exactly one point of $\psi(\omega)$.
		\end{enumerate}
	\end{theorem}
	The second part of item (ii), although not explicitly in \cite{Furstenberg1961} or \cite{JagerKeller2006}, is a consequence of the proof. For completeness, let us prove this part assuming the existence of an invariant $(p,q)$-graph. We follow the notes of \cite{JagerSkewProductNotes}.
\begin{proof}
      Let $\Phi_\omega=\{\phi_1(\omega),\dots,\phi_k(\omega)\}$, with $k=pq$, be cyclically ordered, with $\phi_{k+1}(\omega)=\phi_1(\omega)$. This means that $(\phi_i(\omega),\phi_{i+1}(\omega))\cap \Phi_\omega=\emptyset$ for $1\leq i \leq k$. Observe that, as $g_\omega$ is a homeomorphism and $\Phi_\omega$ is invariant, we have that if $g_\omega(\phi_i(\omega))=\phi_j(g(\omega))$ then $g_\omega(\phi_{i+1}(\omega))=\phi_{j+1}(g(\omega))$.

    Let $\mu\in \cM_m^{\mathrm{erg}}(G)$ and take $\omega\mapsto \mu_\omega$ the Rokhlin disintegration of $\mu$. Observe that since $g$ is invertible, $\mu$ is $G$-invariant, and the Rokhlin disintegration is unique $(g_\omega)_\ast \mu_\omega=\mu_{g(\omega)}$. First, we will prove that $\mu_\omega([\phi_i(\omega),\phi_{i+1}(\omega)))=1/k$ for $m$-almost every $\omega\in \Omega$. Assume not. Then, for each $\omega$ let $\{i_1(\omega)<i_2(\omega)<\dots<i_\ell(\omega)\}$ be such that $\mu_\omega([\phi_{i_j}(\omega),\phi_{i_j+1}(\omega)))<1/k$ for $j=1,\dots,\ell$, with $\ell$ maximal. Observe that for each $\ell$, the set of $\omega$ that satisfies this is $g$-invariant, so there exists $\ell$ that works for $m$-almost every $\omega$.

    For each $\omega$ in this set define $\Phi'_\omega=\{\phi_{i_1(\omega)}(\omega),\dots,\phi_{i_\ell(\omega)}(\omega)\}$. Now observe that as ${g_\omega}_*\mu_\omega=\mu_{g(\omega)}$ and $g_\omega$ preserves orientation, we have that $g_\omega\Phi'_\omega=\Phi'_{g(\omega)}$. Then it is an invariant graph contained in $\Phi_\omega$ and by (MIG1) $\Phi'_\omega=\Phi_\omega$. But this implies that $\mu_\omega(\Sone)<1$, a contradiction.

    Define
\[
r_i(\omega):=
\sup\left\{
r\in[0,\phi_{i+1}(\omega)-\phi_i(\omega)):
\mu_\omega\bigl([\phi_i(\omega),\phi_i(\omega)+r)\bigr)
\leq\frac{1}{2pq}
\right\}
\]
and $\psi_i(\omega):=\phi_i(\omega)+r_i(\omega)$.
This gives an invariant graph $\psi$. Moreover, the map $H_\omega : \Phi_\omega \to \Psi_\omega$ defined by $H_{\omega}(\phi_{i}(\omega))=\psi_{i}(\omega)$ is a measurable bijection and commutes with the dynamics for $m$-a.e. $\omega$, i.e., $H_{g\omega}\circ g_{\omega}=g_\omega \circ H_{\omega}$. In particular, $H:\Phi\to \Psi$ preserves invariant and periodic subgraphs, and also their minimal periods. Hence, $\psi$ is an invariant $(p,q)$-graph with the same integers $p$ and $q$ as $\varphi$.

Now
\[
\mu_\omega\bigl([\phi_i(\omega),\psi_i(\omega))\bigr)
\leq\frac{1}{2pq}
\qquad\text{and}\qquad
\mu_\omega\bigl([\phi_i(\omega),\psi_i(\omega)]\bigr)
\geq\frac{1}{2pq},
\]
and thus
\[
\mu\left(\bigcup_{i=1}^{pq}[\phi_i,\psi_i)\right)
\leq\frac12
\qquad\text{and}\qquad
\mu\left(\bigcup_{i=1}^{pq}[\phi_i,\psi_i]\right)
\geq\frac12.
\]
By ergodicity, we have that the first set has zero measure and the second one has measure $1$, so
$\mu_\omega(\{\psi_i(\omega)\})=1/k$ for $m$-almost every $\omega$. As a consequence $\mu=\mu_\psi$ and $[\phi_i(\omega),\phi_{i+1}(\omega))\cap \Psi(\omega)=\{\psi_i(\omega)\}$.
\end{proof}

\begin{proposition}\label{prop:cyclic-length}
    If the system has an invariant $(p,q)$-graph, then there exists a constant $L>0$ such that, for every finite family of distinct measures
    $$
    \mu^0,\ldots,\mu^{K-1}
    \in\cM_m^{\mathrm{erg}}(G),
    $$
    the measures can be cyclically ordered so that
    $$
    \sum_{r=0}^{K-1}
    W_1(\mu^r,\mu^{r+1})
    \leq L,
    $$
    where $\mu^K:=\mu^0$.
\end{proposition}
\begin{proof}
    Let $\mu^0,\ldots,\mu^{K-1}\in\cM_m^{\mathrm{erg}}(G)$ be a finite family of distinct measures. By Theorem~\ref{thm:classification}, for each $r=0,\ldots,K-1$ there exists an invariant $(p,q)$-graph $\psi^r$ such that $\mu^r=\mu_{\psi^r}$. We take $\psi^0$ as the reference graph and, to simplify notation, write $\varphi^0:=\psi^0$. For $m$-a.e. $\omega$, we reorder (if necessary) the elements in $\Phi_{\omega}^0$ so they are in positive cyclic order and write it as 
    $$
    \Phi_{\omega}^0=\{\phi_1(\omega),\dots,\phi_N(\omega)\},
    $$
    where $N:=pq$ and $\phi_{N+1}=\phi_1$.

    If $\psi^r$ is a different invariant $(p,q)$-graph then, by Theorem~\ref{thm:classification}, for $m$-a.e. $\omega$, each interval $(\phi_j(\omega),\phi_{j+1}(\omega))$ contains exactly one point of $\Psi_\omega^r$. Denote this point by $z_j^r(\omega)$.

    Since the measures $\mu^0,\ldots,\mu^{K-1}$ are distinct, the corresponding graphs are distinct. Hence, by Lemma~\ref{lem:irreducible-graphs}, for $r\neq s$, $\Psi_\omega^r\cap\Psi_\omega^s=\emptyset$ for $m$-a.e. $\omega$. In particular, $z_j^r(\omega)\neq z_j^s(\omega)$ for $m$-a.e. $\omega$ and every $j$.

    For $r,s\in\{1,\ldots,K-1\}$, define on the reference graph $\Phi^0$ the function
    $$
    \chi_{r,s}(\omega,\phi_j(\omega))
    =
    \begin{cases}
    1,&\text{if }z_j^r(\omega)<z_j^s(\omega)
    \text{ in }(\phi_j(\omega),\phi_{j+1}(\omega)),\\
    0,&\text{otherwise}.
    \end{cases}
    $$
    Since the graphs $\varphi^0$, $\psi^r$ and $\psi^s$ are invariant,
    $g_\omega$ maps the reference interval containing $z_j^r(\omega)$ and
    $z_j^s(\omega)$ onto the corresponding interval over $g\omega$. Moreover, it maps these two points to the corresponding points of $\psi^r$ and $\psi^s$. Since $g_\omega$ preserves orientation, their relative order is also preserved and, hence, $\chi_{r,s}\circ G=\chi_{r,s}$ on $\Phi^0$. Since the graph measure $\mu^0$ is ergodic,
    $\chi_{r,s}$ is constant $\mu^0$-almost everywhere. This implies, since $(\mu^0)_\omega=\frac1N\sum_{j=1}^N\delta_{\phi_j(\omega)}$, that there exists $c_{r,s}\in\{0,1\}$ such that, for $m$-almost every $\omega$ and every $j=1,\ldots,N$, $\chi_{r,s}(\omega,\phi_j(\omega))=c_{r,s}$. In particular, for each pair $r,s$, we must have either $z_j^r(\omega)<z_j^s(\omega)$ for $m$-almost every $\omega$ and every $j$, or $z_j^s(\omega)<z_j^r(\omega)$ for $m$-almost every $\omega$ and every $j$. Since there are only finitely many pairs $r,s$, after intersecting the corresponding full-measure sets and relabeling the graphs, we may assume that, for $m$-almost every $\omega$ and every $j$,
    $$
    \phi_j(\omega)<z_j^1(\omega)<\cdots<z_j^{K-1}(\omega)<\phi_{j+1}(\omega).
    $$
    Set $z_j^0(\omega):=\phi_j(\omega)$ and $z_j^K(\omega):=\phi_{j+1}(\omega)$, and notice that $\mu^K=\mu^0$. 

    For $r=0,\ldots,K-1$, we will construct a coupling between $\mu^r$ and $\mu^{r+1}$, where we set $\mu^K:=\mu^0$.
    Recall that $z_j^0(\omega):=\phi_j(\omega)$ and $z_j^K(\omega):=\phi_{j+1}(\omega)$.

    Notice that the fiber measure determined by the points
    $z_1^K(\omega),\ldots,z_N^K(\omega)$ is precisely
    $(\mu^0)_\omega$, since
    $$
    \{\phi_{j+1}(\omega):1\leq j\leq N\}
    =
    \{\phi_j(\omega):1\leq j\leq N\}.
    $$
    Thus the notation $\mu^K=\mu^0$ simply closes the cyclic order.

    For $r=0,\ldots,K-1$, consider
    $$
    \pi^r:= \int_\Omega \frac1N \sum_{j=1}^N \delta_{\bigl((\omega,z_j^r(\omega)),(\omega,z_j^{r+1}(\omega))\bigr)}\,dm(\omega).
    $$
    By construction, the first marginal of $\pi^r$ is $\mu^r$ and the
    second marginal is $\mu^{r+1}$, so $\pi^r$ is a coupling of these two
    measures. On the fiber over $\omega$, this couples the point
    $z_j^r(\omega)$ with $z_j^{r+1}(\omega)$ for every $j=1,\ldots,N$.

    Since $W_1(\mu^r,\mu^{r+1})$ is the infimum of the integral of the
    distance over all couplings, we have
    $$
    \begin{aligned}
    W_1(\mu^r,\mu^{r+1})
    &\leq\int d\bigl((\omega,t),(\omega',t')\bigr)\,d\pi^r\bigl((\omega,t),(\omega',t')\bigr)\\
    &=\frac1N \int_\Omega \sum_{j=1}^N d\left((\omega,z_j^r(\omega)),(\omega,z_j^{r+1}(\omega))\right)\,dm(\omega).
    \end{aligned}
    $$

    Let $\ell([t,t'])$ denote the length of the positively oriented arc
    from $t$ to $t'$ in $\Sone$. By the ordering established above, for
    $m$-almost every $\omega$ and every $j$, $[z_j^r(\omega),z_j^{r+1}(\omega)]
    \subset [\phi_j(\omega),\phi_{j+1}(\omega)]$. Since the two points have the same base coordinate and $\Sone$ is equipped with its geodesic distance, $d\bigl((\omega,t),(\omega,t')\bigr)\leq\ell([t,t'])$.
    
    Taking $t=z_j^r(\omega)$, and $t'=z_j^{r+1}(\omega)$, we obtain
    $$
    d\left((\omega,z_j^r(\omega)),(\omega,z_j^{r+1}(\omega))\right)
    \leq
    \ell\left([z_j^r(\omega),z_j^{r+1}(\omega)]\right).
    $$
    Consequently,
    $$
    W_1(\mu^r,\mu^{r+1})
    \leq
    \frac{1}{N}\int_\Omega
    \sum_{j=1}^N \ell\left([z_j^r(\omega),z_j^{r+1}(\omega)]\right)\,dm(\omega),
    $$
    and summing over $r=0,\ldots,K-1$ yields
    $$
    \sum_{r=0}^{K-1}W_1(\mu^r,\mu^{r+1})\leq\frac{1}{N}\int_\Omega\sum_{j=1}^N\sum_{r=0}^{K-1}\ell\left([z_j^r(\omega),z_j^{r+1}(\omega)]\right)\,dm(\omega).
    $$

    As we observed, for each fixed $\omega$ and $j$, the points $z_j^0(\omega),\ldots,z_j^K(\omega)$ are all ordered inside $[\phi_j(\omega),\phi_{j+1}(\omega)]$. Therefore
    $$
    \sum_{r=0}^{K-1} \ell\left([z_j^r(\omega),z_j^{r+1}(\omega)]\right)=\ell\left([\phi_j(\omega),\phi_{j+1}(\omega)]\right),
    $$
    and summing over $j$ gives
    $$
    \sum_{j=1}^N\ell\left([\phi_j(\omega),\phi_{j+1}(\omega)]\right)=\ell(\Sone).
    $$
    Putting everything together,
    $$
    \begin{aligned}
    \sum_{r=0}^{K-1}W_1(\mu^r,\mu^{r+1})&\leq\frac{1}{N}\int_\Omega\ell(\Sone)\,dm(\omega)\\
    &=
    \frac{1}{N}\ell(\Sone).
    \end{aligned}
    $$
    Setting $L:=\frac{1}{N}\ell(\Sone)$, we obtain $\sum_{r=0}^{K-1}W_1(\mu^r,\mu^{r+1})\leq L$.
\end{proof}

    For a subset $E\subset\Prob(N\times\Sone)$, denote by $N_{W_1}(E,\eps)$ the smallest number of $W_1$-balls of radius $\eps$ needed to cover $E$. We have the following consequence of Proposition~\ref{prop:cyclic-length}:

    \begin{corollary}\label{cor:linear-covering}
    There exists $L>0$ such that, for every $\eps>0$, 
    $$
    N_{W_1}\left(\cM_m^{\mathrm{erg}}(G),\eps\right)
    \leq
    1+\frac{L}{\eps}.
    $$
\end{corollary}
\begin{proof}
If the system has no invariant graph, then by Theorem~\ref{thm:classification} the set $\cM_m^{\mathrm{erg}}(G)$ consists of a single measure, and the result is immediate.

Now suppose that the system has an invariant $(p,q)$-graph and let $\mu^0,\ldots,\mu^{K-1}\in\cM_m^{\mathrm{erg}}(G)$ be distinct and $\eps$-separated measures. By Proposition~\ref{prop:cyclic-length}, after cyclically ordering them, we have
$$
\sum_{r=0}^{K-1}
W_1(\mu^r,\mu^{r+1})
\leq L,
\qquad
\mu^K=\mu^0.
$$
Since the family is $\eps$-separated, $W_1(\mu^r,\mu^{r+1})\geq\eps$ for every $r=0,\ldots,K-2$. Hence
$$
(K-1)\eps
\leq
\sum_{r=0}^{K-2}
W_1(\mu^r,\mu^{r+1})
\leq
L,
$$
i.e., $K\leq1+\frac{L}{\eps}$.

Starting from any measure in $\cM_m^{\mathrm{erg}}(G)$, choose
successively a new measure at $W_1$-distance at least $\eps$ from
all the previously chosen ones, whenever this is possible. The previous
estimate shows that this procedure must stop after a finite number of steps.
By maximality, every measure in $\cM_{m}^{\mathrm{erg}}(G)$ is at distance strictly less than $\eps$ from one of the chosen measures. Therefore, the $\eps$-balls centered at the chosen measures cover $\cM_m^{\mathrm{erg}}(G)$.

This proves that $N_{W_1}\left(\cM_m^{\mathrm{erg}}(G),\eps\right)
\leq 1+\frac{L}{\eps}$.
\end{proof}

%----------------------------------------------------------------------
%----------------------------------------------------------------------
\section{\texorpdfstring{$u$}{u}-Gibbs measures}
%----------------------------------------------------------------------
%----------------------------------------------------------------------

    Recall that $f$ is Axiom A. Then, by the spectral decomposition $\Omega(f)=\Lambda_1\cup \dots \cup \Lambda_k$, where each $\Lambda_i$ is an isolated hyperbolic set. By \cite{Rue76}, there exists a finite number of ergodic SRB measures, $m_{f,1},\dots, m_{f,\ell}$, such that their basins cover a total Lebesgue measure set, moreover, each of these measures is concentrated in one basic piece that is an attractor.
	We denote by $\Leb$ the normalized Lebesgue measure on
	$N\times\Sone$.

For each $1\leq i\leq \ell$, let $B(m_{f,i})$ be the basin of $m_{f,i}$, and define $B_i=\pi^{-1}B(m_{f,i})$. By the absolute continuity of $\pi$, we have that $\cup_{i=1}^\ell B_i$ has total Lebesgue measure.
	\begin{lemma}\label{lem:base-marginal}
		For $z=(x,\theta)\in B_i$, every accumulation point $\mu$ of
		$(\emp_n^F(z))$ satisfies
		\[
		\pi_*\mu=m_{f,i}.
		\]
        In particular, for Lebesgue almost every point $z\in N\times\Sone$, there exists $1\leq i\leq\ell$ such that every accumulation point of $(\emp_n^F(z))_{n\geq1}$ projects to $m_{f,i}$.
	\end{lemma}

	\begin{proof}  
       Take $z\in B_i$, we have
		\[
		\pi_*\emp_n^F(x,\theta)
		=
		\emp_n^f(x)
		\wto
		m_{f,i}.
		\]
		If $\emp_{n_k}^F(x,\theta)\wto\mu$, continuity of pushforward by $\pi$ gives
		$\pi_*\mu=m_{f,i}$.

        As the union of $B_i$ has total Lebesgue measure, we have that for Lebesgue almost every point $z\in N\times \Sone$, there exists $1\leq i \leq \ell$ such that every accumulation point of $(\emp_n^F(z))_{n\geq1}$ projects to $m_{f,i}$.
	\end{proof}

From now on, we fix $i$ and work with points in $B_i$, so for simplicity we omit the index $i$, and we call the SRB of $f$ just by $m$.

\begin{remark}\label{rmk:lambda-ph}
    Let $\Lambda:=\supp(m)$ be the attracting basic piece supporting $m$. Observe that if $U$ is a trapping region of $\Lambda$, $U\times \Sone$ is a trapping region of $\Lambda\times\Sone$, so $\Lambda\times\Sone$ is a partially hyperbolic attracting set for $F$. Moreover, $B(m)\times\Sone$ is contained in the basin of $\Lambda\times\Sone$.
\end{remark}

\begin{remark}\label{rmk:acc-u-gibbs}
Since $F$ is $C^{1+\alpha}$, Theorem~C of \cite{CYZ2020} implies that, for Lebesgue-a.e. point in the basin of $\Lambda\times\Sone$, every accumulation point of the sequence of its empirical measures is a $u$-Gibbs measure whose ergodic components of such a measure are also $u$-Gibbs measures \cite[Corollary~2.15]{CYZ2020}.
\end{remark}

\subsection*{Holonomies and invariance principle}

Let $\Lambda\subset N$ be a basic piece of $f$. For $x\in \Lambda$ we denote by $W_{loc}^s(x)$ the local stable manifold of $x$ with respect to $f$ and $W^s(x)$ the global stable manifold; analogously we denote the unstable ones by changing $s$ to $u$.

Given $p\in \Lambda$ and $x,y\in W_{loc}^s(p)$ we can define the \emph{stable holonomy} $h^s_{x,y}=\lim_{n\to \infty}{f^n_y}^{-1}\circ f^n_x$, by \cite{Avila-Viana-IP} this limit exists and satisfies:
\begin{itemize}
    \item $f_y h^s_{x,y}=h^s_{f(x),f(y)}f_x$,
    \item for $x,y,z\in W^s_{loc}(p)$, $h^s_{x,y}=h^s_{z,y}h^s_{x,z}$
    \item $d(h^s_{x,y},Id)\leq C d(x,y)$,
    \item $\{(x,y,t)\in N\times N\times \Sone;x,y\in W^s_{loc}(z)\text{ for some }z\in \Lambda\}\mapsto h^s_{x,y}(t)$ is uniformly continuous.
\end{itemize}
Analogously, we define the \emph{unstable holonomy} for points $y\in W^u_{loc}(x)$ interchanging $n$ by $-n$.

Given a measure $\mu$ of $N\times \Sone$, that projects to $m$ on $N$, we define the \emph{center disintegration} as the disintegration $x\mapsto \mu_x$ given by the partition $\{\{x\}\times \Sone;x\in N\}$.

We say that its center disintegration is \emph{$s$-invariant} if there exists a set $M^s\subset N$ of full $m$-measure, such that for every $x,y\in M^s$, with $y\in W^s_{loc}(x)$, we have ${h^s_{x,y}}_*\mu_x=\mu_y$.
Analogously, we say that the disintegration is \emph{$u$-invariant} if the same is true interchanging $s$ with $u$. 

Recall that the disintegration is defined for $m$-almost every $x\in \Lambda$. We say that it has a continuous extension on $\Lambda$ if there exists a continuous function $\Lambda\ni x\mapsto \mu_x$, with respect to the weak-$*$ topology, that coincides with the disintegration for $m$-almost every $x$. We say that the continuous extension of the center disintegration of $\mu$ is $su$-invariant if the disintegration is both $s$- and $u$-invariant with $M^s=M^u=\Lambda$.

The following theorem is a particular case of Avila-Viana invariance principle \cite{Avila-Viana-IP}.

\begin{theorem}\label{thm:invariance-principle}
    Let $\mu$ be an ergodic $u$-Gibbs measure; then its center disintegration is $u$-invariant. Moreover, if $\lambda^c(\mu)\geq 0$, then its center disintegration has a continuous extension on $\Lambda$ that is $su$-invariant.
\end{theorem}
\begin{proof}
    By \cite{Tahzibi-Yang} every $u$-Gibbs has a $u$-invariant center disintegration, see also \cite[proposition~2.28]{Obata2023}. Also if $\lambda^c(\mu)\geq 0$, by \cite{Avila-Viana-IP} the disintegration is also $s$-invariant, moreover as $m$ has local product structure, $x\mapsto \mu_x$ has a continuous extension to $\supp(m)=\Lambda$.
\end{proof}
\begin{remark}
    In \cite{Tahzibi-Yang} and \cite{Obata2023} the base dynamics is an Anosov diffeomorphism $C^2$. The regularity hypothesis was only used to apply Ledrappier-Young \cite{LY85b}, but this was extended to $C^{1+}$ regularity in \cite{Brown2022StableHolonomies}. 
    
    Observe that, in our setting, $\Lambda$ is a hyperbolic attractor and therefore forms a $u$-lamination: every unstable manifold is contained in $\Lambda$. Consequently, the proof of \cite{Tahzibi-Yang} applies without modification when the base dynamics is given by $f \colon \Lambda \to \Lambda$.
\end{remark}
%----------------------------------------------------------------------
%----------------------------------------------------------------------
\subsection*{\texorpdfstring{$u$}{u}-Gibbs classification}
%----------------------------------------------------------------------
%----------------------------------------------------------------------

A $(p,q)$-graph $\varphi$ is said to be \emph{continuous} if $\Lambda\ni x\mapsto \Phi_x$ is continuous in the Hausdorff topology. It is said to be \emph{$s$-invariant} if for each $y\in W^s(x)\cap \Lambda$, $h^s_{x,y}(\Phi_x)=\Phi_y$. Analogously, we say that it is \emph{$u$-invariant} if the same is true, changing stable by unstable. If it is both $s$- and $u$-invariant, we say that it is $su$-invariant. We say that a sequence of continuous $(p,q)$-graphs $\varphi_n$ converges uniformly to $\varphi$ if $\sup_{x\in\Lambda} d_H(\Phi^n_x,\Phi_x)\to 0$.

By continuity of the stable and unstable holonomies with respect to base points, we have that every $su$-invariant graph is continuous on $\Lambda$.

\begin{proposition}\label{prop:su-invariant-pq-graphs}
    Let us assume that there exists more than one $u$-Gibbs measure that projects to $m$. Let $\mu$ be an ergodic $u$-Gibbs such that $\pi_*\mu=m$ and $\lambda^c(\mu)\geq 0$. Then there exists a continuous $F$-invariant, $su$-invariant $(p,q)$-graph $\varphi:\Lambda \to (\Sone)^{pq}_*$ such that $\mu=\mu_\varphi$.
\end{proposition}
\begin{proof}
    By Theorem~\ref{thm:classification} we have that there exists $\varphi$ an invariant $pq$-graph such that $\mu=\mu_\varphi$. By Theorem~\ref{thm:invariance-principle} we have that the disintegration of $\mu$, $\mu_x$ is $su$-invariant and continuous, as $\mu_x$ is atomic for $m$-almost every $x$, with $pq$ atoms, then it is atomic for every $x$ with $pq$ atoms, so we can extend $\varphi$ everywhere in the support of $m$.  
    For $y\in W^s(x)\cap \Lambda$ 
    $$
    \begin{aligned}
	(h^s_{x,y})_*\mu_x	&=	\frac1{pq}	\sum_{i=1}^p\sum_{j=1}^q	\delta_{h^s_{x,y}(\varphi_j^i(x))} \\
    &=\mu_y \\
    &=\frac1{pq}	\sum_{i=1}^p\sum_{j=1}^q	\delta_{\varphi_j^i(y)}.
    \end{aligned}
	$$
Analogously for $y\in W^u(x)\cap \Lambda$. So $\varphi$ is $su$-invariant and continuous.
\end{proof}

Moreover, this graph is unique:

\begin{lemma}\label{lem:unique-continuous-graph}
	Let $\varphi$ and $\psi$ be continuous $(p,q)$-invariant graphs
	over $\Lambda$. Then $\mu_\varphi=\mu_\psi$ if and only if $\Phi_x=\Psi_x$ for every $x\in\Lambda$.
\end{lemma}

\begin{proof}
	Suppose first that $\mu_\varphi=\mu_\psi$. By uniqueness of the disintegration over the base, $(\mu_\varphi)_x=(\mu_\psi)_x$ for $m$-a.e.\ $x$. Since both conditional measures are uniform probability measures on $pq$ distinct points, their supports coincide: $\Phi_x=\Psi_x$ for $m$-a.e.\ $x$.

	Since $\supp(m)=\Lambda$, every full $m$-measure subset of
	$\Lambda$ is dense. The maps $x\mapsto\Phi_x$ and
	$x\mapsto\Psi_x$ are continuous in the Hausdorff topology, hence
	the equality extends to every $x\in\Lambda$.

    On the other hand, if $\Phi_x=\Psi_x$ for every $x\in \Lambda$, then 
    $$(\mu_\varphi)_x=\frac1{pq}\sum_{t\in\Phi_x}\delta_t=\frac1{pq}\sum_{t\in\Psi_x}\delta_t=(\mu_\psi)_x,
    $$
    for every $x\in\Lambda$, and integrating w.r.t. $m$ we obtain $\mu_\varphi=\mu_\psi$.
\end{proof}

\begin{proposition}\label{prop:u-closed}
    Let $\mu_n$ be a sequence of ergodic $u$-Gibbs with $\lambda^c(\mu_n)\geq 0$, such that $\mu_n$ converges weakly-$*$ to $\mu$. Then $\mu$ is an ergodic $u$-Gibbs. Moreover, $\mu_{\psi_n}\wto \mu_\psi$ if and only if $\psi_n\to \psi$ uniformly.
\end{proposition}
\begin{proof}
If there is a finite number of ergodic $u$-Gibbs with nonnegative center exponents, the result is direct, so let's assume there are infinitely many.

The projection of $\mu_n$ is an ergodic SRB of $f$. As there are finitely many ones, we have that for $n$ large enough $\mu_n$ projects to a fixed $m$ supported on a basic piece $\Lambda$. Take a periodic point $x\in \Lambda$.

Let $\nu$, be an ergodic $u$-Gibbs with nonnegative center exponent that projects to $m$, and let $\varphi$, its corresponding $pq$-graphs given by Proposition~\ref{prop:su-invariant-pq-graphs}.
Let $N=pq$, order $\Phi_x=\{\varphi_1(x),\dots, \varphi_{N}(x)\}$ cyclically, such that $(\varphi_i(x),\varphi_{i+1}(x))$ does not contain any other point of $\varphi_x$, for $1\leq i\leq N$, with $\varphi_{N+1}(x)=\varphi_1(x)$.

Let $\psi^n$ be the $pq$-graph of $\mu_n$, by Theorem~\ref{thm:classification} we can take $\psi^n_i(x)$ that is contained in $(\varphi_i(x),\varphi_{i+1}(x))$, $1\leq i\leq N$. We have that up to taking a subsequence, $\psi^n_i(x)$ converges to some $\psi_i(x)$ for every $1\leq i\leq N$.

Now there exists $L>0$ such that, for each $z\in \Lambda$ there exists for some $1\leq k \leq per(x)$ and $[z,f^k(x)]\in W^s_L(z)\cap W^u_L(f^k(x))$. Define $\Psi_z=h^s_{[z,f^k(x)],z}h^u_{f^k(x),[z,f^k(x)]}f^k \Psi(x)$, by the $su$-invariance of $\psi^n$, and the fact that holonomies with bounded length are uniformly continuous, $\Psi^n_z\to \Psi_z$ uniformly. In particular $\Psi_z$ does not depend on the choose of $k$ and $[z,f^k(x)]$, moreover $\Psi$ is also $F$ and $su$-invariant.

In principle, $\Psi_x$ may have less than $N$ points. Let's prove that it has exactly $N$. Let $k$ be the number of points. Observe that, by the $su$-invariance, this number does not depend on $x\in \Lambda$. Define $\mu_\psi$ as an invariant measure that projects to $m$ and whose support is contained in $\psi$, this always exists because $\psi$ is invariant and its graph is compact.

Taking the ergodic decomposition of $\mu_\psi$, as $m$ is ergodic, all ergodic components project to $m$. Let $\tilde{\mu}$ be any ergodic component. Taking the Rohklin decomposition on the fibers, we get that $\tilde{\mu}_x$ is atomic with the atoms contained in $\Psi_x$. This defines a graph $\tilde{\psi}$ such that $\tilde{\Psi}_x\subset \Psi_x$. 

Using Theorem~\ref{thm:classification} (with $\varphi$ and $\mu=\tilde{\mu}$) we get that $\tilde{\Psi}_x$ should have $N$ points for almost every $x\in \Lambda$, concluding that $\# \tilde{\Psi}_x=k=N$. Moreover, as $\tilde{\Psi}_x\subset \Psi_x$ have the same number of points we have that $\tilde{\Psi}_x=\Psi_x$, then $\mu_\psi$ is ergodic and $\psi$ is an invariant $pq$-graph.

Now let's prove that $\mu=\mu_\psi$. Take any continuous function $g:N\times \Sone\to \mathbb{R}$ and observe that
$$
\int g \, d \mu_n=
\frac{1}{pq} \int_{N} \sum_{j=1}^{pq}g(x,\psi^n_j(x)) \,dm(x)
$$

By the uniform convergence of $\psi^n\to \psi$ we conclude that $\int g d\mu_n\to \int g d \mu_\psi$. That implies that $\mu_\psi=\mu$ and by Proposition~\ref{prop:pq-measure} we conclude that $\mu$ is ergodic. 

Applying the same argument to any subsequence yields a further subsequence converging uniformly to a continuous invariant graph whose associated measure is $\mu$. By Lemma~\ref{lem:unique-continuous-graph}, we see that the limit $\psi$ is unique, so we get $\psi_n\to \psi$ uniformly.

The fact that $\psi_n\to \psi$ uniformly implies that $\mu_{\psi_n}\wto \mu_\psi$ is straightforward.
\end{proof}

\subsection*{The order of the nonnegative graph measures}

Define
\begin{equation}\label{eq:Ggeq0}
\cG_{\geq 0}:=\left\{\mu:\mu\text{ is an ergodic $u$-Gibbs measure},
\ \pi_*\mu=m,
\ \lambda^c(\mu)\geq 0
\right\}.
\end{equation}

Proposition~\ref{prop:u-closed} has the following corollary:

\begin{corollary}\label{coro:G-compact}
    The set $\cG_{\geq 0}$ is compact in the weak-$*$ topology.
\end{corollary}
\begin{proof}
Take a sequence $(\mu_n)$ of measures in $\cG_{\geq 0}$. By compactness of $\Prob(N\times\Sone)$, we can assume, after passing to a subsequence, that $\mu_n \wto \mu$. By Proposition~\ref{prop:u-closed}, $\mu$ is an ergodic $u$-Gibbs measure.

Since the projection $\pi$ and the push-forward $\pi_\ast$ are continuous, and since $\pi_\ast \mu_n=m$ for every $n$, $\pi_\ast \mu=m$.

Finally, recall that
	$$
	\lambda^c(\nu)
	=
	\int
	\log |f'_x(t)|
	\,d\nu(x,t),
	$$
	and that the integrand is continuous. Hence, $\lambda^c(\mu_n)\to\lambda^c(\mu)$. Since $\lambda^c(\mu_n)\geq0$ for every $n$, we have $\lambda^c(\mu)\geq0$. 
    This shows that $\mu\in\cG_{\geq0}$. 
    \end{proof}

From now on to the end of this section, unless stated otherwise, assume that $\#\cG_{\geq0}\geq 2$. By Proposition~\ref{prop:su-invariant-pq-graphs}, every $\mu\in\cG_{\geq0}$ is of the form $\mu=\mu_\varphi$ for a continuous $F$-invariant, $su$-invariant $(p,q)$-graph $\varphi$. By
Theorem~\ref{thm:classification}, the integers $p$ and $q$ are the same
for all such measures. Set $N:=pq$.

Fix, until the end of this section, a measure $\mu^0\in\cG_{\geq0}$, and let $\varphi^0$ be its corresponding
$(p,q)$-graph. Given a \textcolor{black}{transitive} point $x_0\in \Lambda$, the set $$\Phi^0_{x_0}=\{(\varphi^0)^i_j(x_0);1\leq i\leq p,1\leq j\leq q\}$$ contains exactly $N$ elements. Ordering them in positive cyclic order, we get:
$$\Phi^0_{x_0}=\{a_1,\dots,a_N\},$$ with $a_1<\dots<a_N$.

For every $\mu=\mu_\psi\in\cG_{\geq0}$ with $\mu\neq\mu^0$,
we will show that $\Psi_{x_0}$ contains exactly one point in $(a_1,a_2)$. To do so, we need some lemmas. 

\begin{lemma}\label{lem:graphs-do-not-intersect}
	Let $\varphi$ and $\psi$ be two continuous $F$-invariant, $(p,q)$-graphs over $\Lambda$, and let
	$x_0\in\Lambda$ be the transitive point fixed above. If
	$\Phi_{x_0}\cap\Psi_{x_0}\neq\emptyset$, then $\Phi_x=\Psi_x$ for every $x\in\Lambda$. In particular, if $\varphi\neq\psi$ then $\Phi_{x_0}\cap\Psi_{x_0}=\emptyset$.
\end{lemma}

\begin{proof}
	Suppose $\Phi_{x_0}\cap\Psi_{x_0}\neq\emptyset$ and take $t_0\in\Phi_{x_0}\cap\Psi_{x_0}$. Since $\varphi$ and $\psi$ are $F$-invariant, we have
	\[
	f_{x_0}^n(t_0)
	\in
	\Phi_{f^n(x_0)}\cap\Psi_{f^n(x_0)}
	\]
	for every $n\geq0$. Since we chose $x_0$ to be transitive, its forward orbit is dense in $\Lambda$ and, since both $\varphi$ and $\psi$ are continuous, we have $\Phi_x\cap\Psi_x\neq\emptyset$ for every $x\in\Lambda$.

	By Lemma~\ref{lem:irreducible-graphs}, two irreducible invariant
	graphs are either equal or disjoint on $m$-a.e.\ fiber. Hence
	$\Phi_x=\Psi_x$ for $m$-a.e.\ $x$ and, since
	$\supp(m)=\Lambda$ and both graphs are continuous, the equality
	extends for every $x\in\Lambda$.
\end{proof}

For $su$-invariant graphs we can say more:

\begin{lemma}\label{lem:su-graphs-do-not-intersect}
Let $\varphi$ and $\psi$ be two continuous, $su$-invariant, $F$-invariant $(p,q)$-graphs over $\Lambda$. Then either
$ 
\Phi_x=\Psi_x
$ 
for every $x\in\Lambda$, or
$ 
\Phi_x\cap\Psi_x=\emptyset
$ 
for every $x\in\Lambda$.
\end{lemma}

\begin{proof}
Let $E:=\{x\in\Lambda:\Phi_x\cap\Psi_x\neq\emptyset\}$.
By the previous lemma, it is enough to show that, if $E\neq\emptyset$, then $x_0\in E$.
\
By the density of the orbit of $x_0$, there exists $n$ such that $W^s_{loc}(f^n(x_0))\pitchfork W^u_{loc}(x)=\{z\}$, by the local product structure of $\Lambda$ we have that $z\in \Lambda$, then the $su$-invariance and $F$ invariance implies that $\Phi_{x_0}=f_{x_0}^{-n}h^s_{z,f^n(x_0)} h^u_{x,z}(\Phi_x)$. So $x\in E$ if and only if $x_0\in E$. 
\end{proof}

The next lemma tells us how points of different graphs are ordered in a fiber:

\begin{lemma}\label{lem:interlacing-at-x0}
    Let $\varphi$ and $\psi$ be two distinct continuous
	$F$-invariant $(p,q)$-graphs over $\Lambda$. Let $a,b\in\Phi_{x_0}$ be two consecutive points in $\Phi_{x_0}$ in positive cyclic order. Then $\#\bigl(\Psi_{x_0}\cap(a,b)\bigr)=1$.
\end{lemma}

\begin{proof}
	First observe that, by item $(ii)$ of Theorem~\ref{thm:classification}, for $m$-a.e.\ $x$, each connected component of $\Sone\setminus\Phi_x$ contains exactly one point of $\Psi_x$. Also note that, by
	Lemma~\ref{lem:irreducible-graphs}, since $\varphi\neq \psi$, we have $\Phi_x\cap\Psi_x=\emptyset$ for $m$-a.e.\ $x$. Since $\supp(m)=\Lambda$, we can choose a sequence $(x_n)$ in $\Lambda$ such that $x_n\to x_0$ and for which both properties hold simultaneously.

	By continuity of $\varphi$ in the Hausdorff topology, for $n$ sufficiently large there are consecutive points $a_n,b_n\in\Phi_{x_n}$ such that $a_n\to a$, and $b_n\to b$. Moreover, for each such $n$, there exists a unique $z_n\in\Psi_{x_n}\cap(a_n,b_n)$.

	Passing to a subsequence if necessary, we may assume that
	$z_n\to z$ and, by continuity of $\psi$ in the Hausdorff topology, we have $z\in\Psi_{x_0} \cap [a,b]$. In fact, since $\varphi$ and $\psi$ are distinct, we have that $\Phi_{x_0}\cap\Psi_{x_0}=\emptyset$ by Lemma~\ref{lem:graphs-do-not-intersect}, so $z\neq a,b$, and hence $z\in(a,b)$. This proves that $\#\bigl(\Psi_{x_0}\cap(a,b)\bigr)\geq1$.

	Suppose now that there are two distinct points $z,z'\in\Psi_{x_0}\cap(a,b)$. Choose disjoint neighborhoods $U$ and $U'$ of $z$ and $z'$ whose closures are contained in $(a,b)$. By continuity of $\psi$ in the Hausdorff topology, for all sufficiently large $n$ the set
	$\Psi_{x_n}$ contains a point in $U$ and a point in $U'$.
	Since also $a_n\to a$ and $b_n\to b$, both of these points belong
	to the interval $(a_n,b_n)$ for $n$ sufficiently large. This contradicts the fact that $\#\bigl(\Psi_{x_n}\cap(a_n,b_n)\bigr)=1$ for every $n$ and shows that $\#\bigl(\Psi_{x_0}\cap(a,b)\bigr)=1$.
\end{proof}

If $\mu=\mu_\psi\in\cG_{\geq0}$ with $\mu\neq\mu^0$, then $\psi\neq \varphi^0$. Then, Lemma~\ref{lem:interlacing-at-x0} implies that the set $\Psi_{x_0}\cap(a_1,a_2)$ contains exactly one point. We denote this point by $z(\mu)$.

Now, let $I:=[a_1,a_2]$ and denote by $\widehat I:=I/\{a_1\sim a_2\}$ the quotient obtained by identifying the two endpoints. Consider the map $\iota:\cG_{\geq0}\to\widehat I$ defined by
\begin{equation}\label{eq:iota}
\iota(\mu)
=
\begin{cases}
z(\mu),&\mu\neq\mu^0,\\
[a_1]=[a_2],&\mu=\mu^0.
\end{cases}
\end{equation}

\begin{proposition}\label{prop:iota-well-defined}
The map $\iota:\cG_{\geq0}\to\widehat I$ is well-defined and injective.
\end{proposition}

\begin{proof}
Let $\mu\in\cG_{\geq0}$ be a measure with $\mu\neq\mu^0$, and suppose that $\mu=\mu_\psi=\mu_{\widetilde\psi}$ for two continuous $(p,q)$-graphs $\psi$ and $\widetilde\psi$. By Lemma~\ref{lem:unique-continuous-graph}, $\Psi_x=\widetilde\Psi_x$ for every $x\in\Lambda$. In particular, $\Psi_{x_0}=\widetilde\Psi_{x_0}$, and both graphs must determine the same unique point in $(a_1,a_2)$. This shows that $\iota(\mu)$ depends only on $\mu$, and, hence, that $\iota$ is well-defined.

For the injectivity part, let $\mu,\nu \in \cG_{\geq0}$ and suppose that $\iota(\mu)=\iota(\nu)$. First, if $\mu=\mu^0$, then $\iota(\mu)=[a_1]=[a_2]$. If $\nu\neq \mu^0$, then $z(\nu)\in(a_1,a_2)$ and $\iota(\mu)\neq\iota(\nu)$, showing that $\mu=\mu^0=\nu$.

Now, suppose both $\mu$ and $\nu$ are different from $\mu^0$, and that their corresponding graphs are $\psi$ and $\varphi$, respectively. Then, $z(\mu)=z(\nu)\in \Psi_{x_0}\cap\Phi_{x_0}$. By Lemma~\ref{lem:graphs-do-not-intersect}, $\Psi_x=\Phi_x$ for every $x\in\Lambda$ and, by Lemma~\ref{lem:unique-continuous-graph}, $\mu=\nu$.
\end{proof}

We now show that $\iota$ is continuous. Since it is injective, $\cG_{\geq0}$ is compact and $\widehat I$ is Hausdorff, we can say more:

\begin{proposition}\label{prop:iota-homeo}
    The map $\iota:\cG_{\geq0}\to\widehat I$ is a homeomorphism onto its image.
\end{proposition}
\begin{proof}
    Since $\widehat I$ is Hausdorff, $\cG_{\geq0}$ is compact by Corollary~\ref{coro:G-compact}, and $\iota$ is injective by Proposition~\ref{prop:iota-well-defined}, it suffices to show that $\iota$ is continuous.

    To prove continuity, take $(\mu_n)_n$ a sequence in $\cG_{\geq0}$ with $\mu_n \to \mu$. We will show that $\iota(\mu_n)\to \iota(\mu)$. Write $\mu_n=\mu_{\psi_n}$ and $\mu=\mu_\psi$ and recall that, by Proposition~\ref{prop:u-closed}, we have $\psi_n\to \psi$ uniformly and, in particular, $(\Psi_n)_{x_0}\to \Psi_{x_0}$ in the Hausdorff topology. Then, if $\mu\neq \mu_0$, $\{z(\mu_n)\}=(\Psi_n)_{x_0}\cap [a_1,a_2]\to \Psi_{x_0}\cap [a_1,a_2]=\{z(\mu)\}$, so $z(\mu_n)\to z(\mu)$. If $\mu=\mu_0$, $\Psi_{x_0}\cap [a_1,a_2]=\{a_1,a_2\}$, then any accumulation point of $z(\mu_n)$ is equal to $a_1$ or $a_2$. In any case $\iota(\mu_n)\to \iota(\mu).$ 
\end{proof}

Now let's prove the key lemma to understand the set of empirical measures with zero center exponent.
\begin{lemma}\label{lem:empirical-segment}
Let $(x,t)\in N\times\Sone$ be such that $x\in B(m)$. Assume that, for every accumulation point $\eta$ of the sequence
$\bigl(\emp_n(x,t)\bigr)_{n\geq1}$, the ergodic decomposition of
$\eta$ is supported on $\cG_{\geq0}$. Then there exist $\mu^-,\mu^+\in\cG_{\geq0}$ such that
$\acc\bigl(\emp_n(x,t)\bigr) \subset [\mu^-,\mu^+]$, where $[\mu^-,\mu^+]:=\left\{(1-s)\mu^-+s\mu^+:s\in[0,1]\right\}$.
\end{lemma}
\begin{proof}
First, assume that $x\in \Lambda$, the attractor that supports $m$. Since $x\in B(m)$ and $\supp(m)=\Lambda$, the forward orbit of $x$ is dense in $\Lambda$.

If $(x,t)$ belongs to the graph $\varphi$ corresponding to some $\mu=\mu_{\varphi} \in \cG_{\geq 0}$, then every accumulation point $\eta$ of the empirical measures of $(x,t)$ is equal to $\mu$, and we can take $\mu^-=\mu^+=\mu$. Indeed, since the graph $\varphi$ is invariant, the forward orbit of $(x,t)$ is contained in it, so every accumulation measure is supported on $\operatorname{graph}(\varphi)$. Since $x\in B(m)$, $\pi_\ast \eta =m$. Finally, by ergodic decomposition and Theorem~\ref{thm:classification}, every ergodic component of $\eta$ must be $\mu_{\varphi}$, so that $\eta$ must be $\mu_{\varphi}$ itself. So, from now on, assume that $(x,t)$ does not belong to the graph of any measure in $\cG_{\geq 0}$. 

Since $t\notin \Phi^0_{x}$, after relabeling $\Phi^0_{x}$ in positive cyclic order, we may assume that $t\in (a_1(x),a_2(x))$. Now, let $I_{x}=[a_1(x),a_2(x)]$ and $\iota_{x}:\cG_{\geq 0}\to \widehat I_{x}$ defined above, using the transitive point $x$ (instead of $x_0$). By Proposition~\ref{prop:iota-homeo}, the map $\iota_{x}$ is a homeomorphism onto its image. Since $\cG_{\geq 0}$ is compact, $S_{x}:=\iota_x (\cG_{\geq 0})$ is a compact subset of $\widehat I_{x}$. 

Notice that, since $(x,t)$ does not belong to any graph, $t\notin S_{x}$. Let $J$ be the connected component of $\widehat{I}_{x}\setminus S_x$ containing $t$. By compactness of $S_x$, both endpoints of $J$ belong to $S_x$, so that there exist two unique measures $\mu^-,\mu^+\in \cG_{\geq 0}$ corresponding to these endpoints, with corresponding invariant graphs $\psi^-$ and $\psi^+$ (notice that since $\#\cG_{\geq 0}\geq 2$, we have $\mu^-\neq\mu^+$). Since $J$ is contained in $\widehat{I}_{x}\setminus S_x$ we can associate it to an open interval in $(a_1(x),a_2(x))$ and write it as $J=(t^-,t^+)\subset (a_1(x),a_2(x))$. Then, $t\in (t^-,t^+)$, $\iota_x(\mu^-)=t^-$, and $\iota_x(\mu^+)=t^+$. Also, $t^-\in \Psi^-_x$ and $t^+\in \Psi^+_x$, and there is no continuous invariant $(p,q)$-graph corresponding to a measure in $\cG_{\geq 0}$ which intersects $\{x\}\times (t^-,t^+)$.

Define $$K=\overline{\cup_{n\geq 0} F^n\left(\{x\} \times [t^-,t^+] \right)}.$$ Notice that, since the fiber maps preserve orientation and $t\in [t^-,t^+]$, the forward orbit of $(x,t)$ by $F$ is contained in $K$ and, as the forward orbit of $x$ by $f$ is dense in $\Lambda$, we have $\pi(K)=\Lambda$.

\begin{claim}
        Let $\varphi$ be a continuous invariant $pq$-graph corresponding to a measure $\cG_{\geq 0}$, different from $\psi^-$ and $\psi^+$. Then $\operatorname{graph}(\varphi)\cap K=\emptyset$.
\end{claim}
\begin{proof}
    Assume that there exists $(z,s)\in \operatorname{graph}(\varphi)\cap K$. By definition of $K$, there exist a sequence of integers $n_k\geq 0$ and points $s_k\in f_x^{n_k}([t^-,t^+])$ such that $f^{n_k}(x)\to z$ and $s_k\to s$.

    Since, by hypothesis, $\varphi$ is different from $\psi^-$ and $\psi^+$, Lemma~\ref{lem:su-graphs-do-not-intersect} implies that $\Phi_z \cap \Psi_z^-=\emptyset$ and $\Phi_z \cap \Psi_z^+=\emptyset$. In particular, $s\notin \Psi_z^-\cup\Psi_z^+$.

    Now assume, by contradiction, that $d\bigl(s_k,f_x^{n_k}(t^-)\bigr)$ is not bounded away from zero. Then, after passing to a subsequence, $d\bigl(s_k,f_x^{n_k}(t^-)\bigr)\to0$. Since $\psi^-$ is invariant $f_x^{n_k}(t^-)\in\Psi^-_{f^{n_k}(x)}$; since it is continuous, $f^{n_k}(x)\to z$, and $s_k\to s$, we have $s\in\Psi^-_z$. But $s\in\Phi_z$, and this contradicts $\Phi_z\cap\Psi^-_z=\emptyset$. This shows that $d\bigl(s_k,f_x^{n_k}(t^-)\bigr)$ is indeed bounded away from zero, and the same argument shows that $d\bigl(s_k,f_x^{n_k}(t^+)\bigr)$ also is. Therefore, there exists $\delta>0$ such that, for all $k$ large enough,
	we have $d\bigl(s_k,f_x^{n_k}(t^-)\bigr)\geq\delta$ and $d\bigl(s_k,f_x^{n_k}(t^+)\bigr)\geq\delta$.

    Take $\beta>0$ such that, if $d(z,y)<\beta$, then $d_H(\Phi_z,\Phi_y)<\frac{\delta}{10}$. Take $k$ sufficiently large so that $d(z,f^{n_k}(x))<\beta$ and $d(s_k,s)<\frac{\delta}{10}$. Since $s\in\Phi_z$, there exists $r_k\in\Phi_{f^{n_k}(x)}$ such that $d(r_k,s)<\frac{\delta}{10}$, and hence $d(r_k,s_k)<\frac{\delta}{5}$. Since $s_k$ is at distance at least $\delta$ from both endpoints $f_x^{n_k}(t^-)$ and $f_x^{n_k}(t^+)$, we conclude that $r_k\in f_x^{n_k}\bigl((t^-,t^+)\bigr)$. On the other hand, since $r_k\in\Phi_{f^{n_k}(x)}$ and $\varphi$ is invariant, there exists $r\in\Phi_x$ such that $f_x^{n_k}(r)=r_k$ and, since $f_x^{n_k}$ is a homeomorphism, $r\in(t^-,t^+)$. This shows that $\Phi_x\cap(t^-,t^+)\neq\emptyset$, contradicting the choice of $t^-$ and $t^+$. 
\end{proof}

    Since the forward orbit of $(x,t)$ is contained in $K$, $\emp_n(x,t)(K^c)=0$, for every $n$. If $\eta$ is an accumulation point of $\emp_n(x,t)$, Portmanteau theorem gives $\eta(K^c)=0$. Indeed, if $\emp_{n_k}(x,t)\wto \eta$ then, since $K^c$ is open, Portmanteau implies: $$\eta(K^c)\leq \liminf_{k\to \infty} \emp_{n_k}(x,t)(K^c)=0.$$
    
    Consider the ergodic decomposition $\eta=\int \eta_e\,d\widetilde\nu(e)$. By hypothesis, $\eta_e\in\cG_{\geq0}$ for $\widetilde\nu$-almost every $e$.

    Take an ergodic component $\eta_e$ different from $\mu^-$ and $\mu^+$. Since $\eta_e\in\cG_{\geq0}$, by Proposition~\ref{prop:su-invariant-pq-graphs} there exists a continuous invariant $(p,q)$-graph $\varphi$ such that $\eta_e=\mu_\varphi$. By the claim, $\operatorname{graph}(\varphi)\cap K=\emptyset$, and, since $\mu_\varphi$ is supported on $\operatorname{graph}(\varphi)$, $\eta_e(K^c)=1$. On the other hand,
    $$
    0=\eta(K^c)=\int \eta_e (K^c)\ d\tilde{\nu}(e).
    $$ Since the integrand is nonnegative and is equal to $1$ for every ergodic component different from $\mu^-$ and $\mu^+$, it follows that $\widetilde\nu$ is supported on $\{\mu^-,\mu^+\}$. In particular, there exists $s\in [0,1]$ such that $\eta=(1-s)\mu^-+s\mu^+$, i.e., $\eta \in [\mu^-,\mu^+]$. Since $\eta$ was an arbitrary accumulation point, we conclude that $\acc\bigl(\emp_n(x,t)\bigr) \subset [\mu^-,\mu^+]$.

    Now suppose $x\notin \Lambda$. Since $x\in B(m)$ and $\Lambda=\supp(m)$ is an attractor, for $n$ large enough there exists $x'\in \Lambda$ such that $f^n(x)\in W^s_{loc}(x')$. Taking $t'=h^s_{f^n(x),x'}(f^n_x(t))$, we have that $(x',t')$ belong to the same stable leaf of $F$. In particular, their forward orbits are asymptotic, and we have
    $$\acc\bigl(\emp_n(x,t)\bigr)=\acc\bigl(\emp_n(x',t')\bigr).$$
    Since we also have $x'\in B(m)$, we can apply the previous case to $(x',t')$ to get $\mu^-,\mu^+\in\cG_{\geq0}$ such that $\acc\bigl(\emp_n(x,t)\bigr) \subset [\mu^-,\mu^+]$, as desired.
\end{proof}

%----------------------------------------------------------------------
%----------------------------------------------------------------------
\section{\texorpdfstring{$W_1$}{W1}-lengths and covering numbers}\label{section:lengths and covering numbers}
%----------------------------------------------------------------------
%----------------------------------------------------------------------

In our setting, for Lebesgue almost every $(x,t)\in B(m)\times\Sone$, every accumulation point of $(\emp_n(x,t))_n$ is a $u$-Gibbs measure projecting to $m$. In the present section, we study metric properties of objects associated with the sets $\cG_{\geq 0}$ defined in (\ref{eq:Ggeq0}) and
\[
\cG_{< 0}:=\left\{\mu:\mu\text{ is an ergodic $u$-Gibbs measure},
\ \pi_*\mu=m,
\ \lambda^c(\mu)< 0
\right\}
\]
that will help us to bound the emergence of the system. 

The first observation we need to make is that, due to Corollary~\ref{cor:linear-covering}, there exists $L>0$ such that for every $\eps>0$ we have $N_{W_1}\left(\cG_{<0},\eps\right)
\leq 1+\frac{L}{\eps}$ and $N_{W_1}\left(\cG_{\geq 0},\eps\right)
\leq 1+\frac{L}{\eps}$. As we will see below, the first estimate will be enough for us in the case of negative exponents, but for nonnegative center exponents, the second is not enough. For this, we need to study in further detail how Theorem~\ref{thm:classification} embeds $\cG_{\geq 0}$ inside $\Prob(N\times\Sone)$.

Let $\iota$ be the homeomorphism defined in~\eqref{eq:iota} and let $S:=\iota(\cG_{\geq0})\subset\widehat I$. Recall that $S$ is a compact set, but need not be the circle $\widehat I$. Nevertheless, we can extend $\iota^{-1}$ to a continuous map $\Gamma:\widehat I\to \Prob(N\times\Sone)$ as follows. 

For $s\in S$, set $\Gamma|_S(s):=\iota^{-1}(s)$. If $J$ is a connected component of $\widehat I\setminus S$, let $a_J,b_J\in S$ be its endpoints and set $\mu_J^-:=\iota^{-1}(a_J)$ and $\mu_J^+:=\iota^{-1}(b_J)$. Let $r_J:J\to (0,1)$ be the affine parametrization of $J$ such that $r_J(s)\to0$ as $s\to a_J$ and $r_J(s)\to1$ as $s\to b_J$. We also denote by $r_J$ its continuous extension to $\overline J$, so that $r_J(a_J)=0$, and $r_J(b_J)=1$. Define $\Gamma|_J(s):=(1-r_J(s))\mu_J^-+r_J(s)\mu_J^+$, for $s\in J$. Thus, on every complementary component $J$, the map $\Gamma$ parametrizes the segment joining $\mu_J^-$ and $\mu_J^+$.

\begin{lemma}\label{lemma:Gamma-cont}
	The map $\Gamma:\widehat I\to\Prob(N\times\Sone)$ is continuous.
\end{lemma}

\begin{proof}
    To prove continuity, first observe that $\Gamma|_S=\iota^{-1}$ is already continuous, since $\iota$ is a homeomorphism onto $S$. Moreover, if $J$ is a connected component of $\widehat I\setminus S$, then $\Gamma$ is also continuous on $J$: for $s,s'\in J$, writing $r=r_J(s)$ and $r'=r_J(s')$, the convexity of $W_1$ gives $W_1\bigl(\Gamma(s),\Gamma(s')\bigr) \leq |r-r'|\,W_1(\mu_J^-,\mu_J^+)$, and the same estimate shows that $\Gamma(s)\to\mu_J^-=\Gamma(a_J)$ as $s\to a_J$, and $\Gamma(s)\to\mu_J^+=\Gamma(b_J)$ as $s\to b_J$. This shows that $\Gamma$ is continuous not only in $J$, but on $\overline J$.

	It remains to check continuity at points of $S$ at which different connected components of $\widehat I\setminus S$ may accumulate. Let $s_n\to s$ with $s\in S$. 
    
    If for all $n$, $s_n\in S$, the conclusion follows from
	the continuity of $\iota^{-1}$ on $S$. So we may consider a subsequence for which $s_n\notin S$, and let $J_n=(a_n,b_n)$ be the connected component of $\widehat I\setminus S$ containing $s_n$.

	Suppose first that, after passing to a subsequence, all the $J_n$
	are equal to a fixed component $J$. Since $s_n\to s\in S$, the
	point $s$ is one of the endpoints of $J$ and the continuity of
	$\Gamma$ on $\overline J$, proved above, gives $\Gamma(s_n)\to\Gamma(s)$.

	Therefore, the truly new case is when the components $J_n$ are pairwise distinct. Since they are pairwise disjoint open arcs of the (compact) circle $\widehat I$, we have $\operatorname{diam}(J_n)\to0$. Since $s_n\in J_n$ and $s_n\to s$, it follows that $a_n\to s$, $b_n\to s$. Then, as $a_n,b_n,s\in S$ and $\Gamma|_S=\iota^{-1}$ is continuous,
	we obtain $\Gamma(a_n)\to\Gamma(s)$ and $\Gamma(b_n)\to\Gamma(s)$.

	Since $s_n \in J_n$, for some $r_n\in(0,1)$ we have that $\Gamma(s_n)=(1-r_n)\Gamma(a_n)+r_n\Gamma(b_n)$. By convexity of $W_1$,
	$$
	W_1\bigl(\Gamma(s_n),\Gamma(s)\bigr)
	\leq
    (1-r_n)W_1\bigl(\Gamma(a_n),\Gamma(s)\bigr)+r_nW_1\bigl(\Gamma(b_n),\Gamma(s)\bigr),
	$$
	and the right-hand side converges to zero. Therefore $\Gamma(s_n)\to\Gamma(s)$, and this proves the continuity of $\Gamma$ on all of $\widehat I$.
\end{proof}

%==================================================================
\subsection{A covering bound for \texorpdfstring{$\Gamma(\widehat{I})$}{GammaI}}\label{subsection:bound for Gamma(I)}
%==================================================================

 The main goal of this section is to bound the number of $W_1$-balls of radius $\eps>0$ that we need to cover $\Gamma(\widehat{I})$.

\begin{theorem}\label{thm:Gamma-covering}
	There exists a constant $L>0$ such that, for every $\eps>0$, $N_{W_1}\bigl(\Gamma(\widehat I),\eps\bigr) \leq 1+\frac{L}{\eps}$.
\end{theorem}

The theorem follows once we prove two lemmas. Before doing so, it will be convenient to cut the circle $\widehat I$ at the point $[a_1]=[a_2]=\iota(\mu^0)$. Let $q:I=[a_1,a_2]\to\widehat I$ be the quotient map. From now on, by a slight abuse of notation, we identify $S$ with its preimage $q^{-1}(S)\subset I$, and we also denote by $\iota^{-1}$ the map $\iota^{-1}:S\to\Prob(N\times\Sone)$, defined as $s\mapsto \iota^{-1}(q(s))$. Notice that, in particular, $\iota^{-1}(a_1)=\iota^{-1}(a_2)=\mu^0$.

Similarly, we will identify $\Gamma$ with the map $\Gamma\circ q:I\to\Prob(N\times\Sone)$, and we will continue to denote it by $\Gamma$.

Thus, from now on, we regard $S$ as a compact subset of the interval
$I=[a_1,a_2]$ containing both endpoints, and $\iota^{-1}$ and $\Gamma$
as maps defined on $S$ and $I$, respectively. Notice that this identification does not change the image of $\Gamma$ in $\Prob(N\times\Sone)$, since $q$ is one-to-one on $(a_1,a_2)$ and the two endpoints satisfy $\Gamma(q(a_1))=\Gamma(q(a_2))=\mu^0$. In particular, $\Gamma(I)=\Gamma(\widehat I)$, and this change of parametrization does not affect the covering numbers considered below.

Recall that, if $E\subset[a,b]$ and $\gamma: E\to X$ is a map from $E$ into a metric space $(X,d)$, the length of $\gamma$ is defined as
\[
\operatorname{Len}_d(\gamma)
:=
\sup
\left\{
\sum_{j=0}^{m-1}
d\bigl(\gamma(s_j),\gamma(s_{j+1})\bigr)
:
s_0<\cdots<s_m,\quad s_j\in E
\right\}.
\]
Of course, we say that $\gamma$ has finite length if $\operatorname{Len}_d(\gamma)<\infty$.

\begin{lemma}\label{lemma:finite-length}
The map $\iota^{-1}$ has finite $W_1$-length in $\bigl(\Prob(N\times\Sone),W_1\bigr)$, i.e., there exists $L>0$ such that, for every finite family $s_0<\cdots<s_m$, with $s_j\in S$, we have
\[
\sum_{j=0}^{m-1}W_1\left(\iota^{-1}(s_j),\iota^{-1}(s_{j+1})\right)\leq L.
\]
\end{lemma}

\begin{proof}
Take a finite family $s_0<\cdots<s_m$, with $s_k\in S$, and set $\mu^k:=\iota^{-1}(s_k)$, for $k=0,\ldots,m$. Recall that, if $s_k\in(a_1,a_2)$, then by definition of $\iota$, $s_k=z(\mu^k)$ (where $z(\mu^k)$ is the unique point of the graph corresponding to $\mu^k$ contained in the reference interval $(a_1,a_2)$). Hence, the order
$$
s_0<\cdots<s_m
$$
is precisely the order induced by the positive cyclic order of the
corresponding graphs in the reference interval $(a_1,a_2)$. This is the
cyclic order used in Proposition~\ref{prop:cyclic-length}.

Suppose first that the family does not contain both endpoints $a_1$
and $a_2$. Then the measures $\mu^0,\ldots,\mu^m$ are all distinct. Proposition~\ref{prop:cyclic-length}, applied to this
finite family with the cyclic order described above, gives
$$
\sum_{k=0}^{m-1}
W_1(\mu^k,\mu^{k+1})
\leq L,
$$
since the sum on the left is obtained from the cyclic sum in that
proposition by omitting the last term joining $\mu^m$ back to
$\mu^0$.

Suppose now that the family contains both endpoints $a_1,a_2$ of the reference interval. Since the points are ordered, we have $s_0=a_1$, and $s_m=a_2$. By our convention, after cutting open $\widehat I$,
$\mu^0=\iota^{-1}(a_1)
=
\iota^{-1}(a_2)=\mu^m.
$
Then, the measures $\mu^0,\ldots,\mu^{m-1}$ are distinct, and in this case
$$
\sum_{k=0}^{m-1}
W_1(\mu^k,\mu^{k+1})
$$
is exactly the cyclic sum associated to this family. Therefore,
Proposition~\ref{prop:cyclic-length} again gives
$$
\sum_{k=0}^{m-1}
W_1(\mu^k,\mu^{k+1})
\leq L.
$$

Therefore, in either case,
$$
\sum_{k=0}^{m-1}
W_1\left(
\iota^{-1}(s_k),\iota^{-1}(s_{k+1})
\right)
\leq L.
$$
Since the finite family $s_0<\cdots<s_m$ in $S$ was arbitrary, we
conclude that $\operatorname{Len}_{W_1}(\iota^{-1})\leq L<\infty$.
\end{proof}

\begin{lemma}\label{lemma:Gamma-finite-length}
The $W_1$-length of the map $\Gamma:I\to\Prob(N\times\Sone)$ satisfies $\operatorname{Len}_{W_1}(\Gamma)\leq\operatorname{Len}_{W_1}(\iota^{-1})$. In particular, $\Gamma$ has finite $W_1$-length.
\end{lemma}
\begin{proof}
    Let $s_0<\dots<s_m$ be a finite family of points $s_k\in I$. Our goal is to bound the sum $\sum_{k=0}^{m-1} W_{1}(\Gamma(s_k),\Gamma(s_{k+1}))$ uniformly over all such families.

    First, observe that if two consecutive points of the family, say $s_k,s_{k+1}$, lie in the same connected component $J=(a_J,b_J)$ of $I\setminus S$, then, if we write $\Gamma(s)=(1-r_J(s))\mu_J^-+r_J(s)\mu_J ^+$, we obtain 
    $$
    W_{1}(\Gamma(s_k),\Gamma(s_{k+1}))\leq |r_J(s_{k+1})-r_J(s_k)|W_1(\mu_J^-,\mu_J^+),
    $$
    where $\mu_J^-=\iota^{-1}(a_J)=\Gamma(a_J)$ and $\mu_J^+=\iota^{-1}(b_J)=\Gamma(b_J)$. In fact, if $s_{k_0}<\dots<s_{k_r}$ are consecutive points of the family all in the same interval $\overline{J}$, then 
    $$
    \sum_{i=0}^{r-1}W_{1}(\Gamma(s_{k_i}),\Gamma(s_{k_{i+1}}))\leq W_1(\Gamma(a_J),\Gamma(b_J)),
    $$
    since $\sum_{i=0}^{r-1}|r_{J}(s_{k_i})-r_{J}(s_{k_{i+1}})|=r_{J}(s_{k_r})-r_{J}(s_{k_{0}})\leq 1$.

    Observe that it is enough to consider points in $S$. Start by adding to the family $s_0<\dots<s_m$ the endpoints $a_J,b_J$ of every
    connected component $J$ of $I\setminus S$ which contains at least
    one of the points $s_k$. Since the original family is finite, only
    finitely many points are added. By adding points to the family, we can only increase the corresponding sum of consecutive $W_1$-distances. So if we can bound this increased family, we can bound the previous one. For each such component $J$, the previous estimate shows that the whole contribution of the points contained in $\overline J$ is bounded by $W_1(\Gamma(a_J),\Gamma(b_J))$. Hence, after adding the endpoints, we may discard all the points in the interior of $J$ when looking for an upper bound.

    After doing this for every connected component met by the family,
    we obtain a finite ordered family $u_0<\dots<u_q$, with $u_i\in S$,
    such that
    $$
    \sum_{k=0}^{m-1}
    W_1(\Gamma(s_k),\Gamma(s_{k+1}))
    \leq
    \sum_{i=0}^{q-1}
    W_1(\Gamma(u_i),\Gamma(u_{i+1})).
    $$
    
    Since every $u_i\in S$, $\Gamma(u_i)=\iota^{-1}(u_i)$ for $i=0,\dots,q$. So, $$\sum_{i=0}^{q-1}W_1(\Gamma(u_i),\Gamma(u_{i+1}))=\sum_{i=0}^{q-1}W_1(\iota^{-1}(u_i),\iota^{-1}(u_{i+1}))\leq \operatorname{Len}_{W_1}(\iota^{-1}),$$ and then $\sum_{k=0}^{m-1} W_1(\Gamma(s_k),\Gamma(s_{k+1}))\leq \operatorname{Len}_{W_1}(\iota^{-1})$.

    Noticing that this holds for every finite family $s_0<\cdots<s_m$ in $I$ and taking the supremum gives $\operatorname{Len}_{W_1}(\Gamma)\leq \operatorname{Len}_{W_1}(\iota^{-1})$. By Lemma~\ref{lemma:finite-length}, $\Gamma$ has finite $W_1$-length.
\end{proof}

We are finally ready to prove Theorem~\ref{thm:Gamma-covering}. Before getting into it, let $L:=\operatorname{Len}_{W_1}(\Gamma)<\infty$ and recall that we have $\Gamma(I)=\Gamma(\widehat I)$.

\begin{proof}[Proof of Theorem~\ref{thm:Gamma-covering}]
    Fix $\eps>0$. Since $I$ is compact and $\Gamma$ is continuous, $\Gamma(I)$ is compact. Let $\{\nu_1,\ldots,\nu_q\}\subset\Gamma(I)$ be a maximal $\eps$-separated set, i.e., $W_1(\nu_i,\nu_j)\geq\eps$ for every $i\neq j$.

    Now, for each $i$, choose $s_i\in I$ such that $\Gamma(s_i)=\nu_i$ and, after relabeling if necessary, assume that $s_1<\cdots<s_q$. Since $L=\operatorname{Len}_{W_1}(\Gamma)$, we have that $\sum_{i=1}^{q-1}
    W_1\bigl(\Gamma(s_i),\Gamma(s_{i+1})\bigr)\leq L$, so 
    $$
    (q-1)\eps\leq \sum_{i=1}^{q-1} W_1(\nu_i,\nu_{i+1}) = \sum_{i=1}^{q-1} W_1\bigl(\Gamma(s_i),\Gamma(s_{i+1})\bigr) \leq L,
    $$
    i.e., $(q-1)\eps\leq L$, or equivalently, $q\leq1+\frac{L}{\eps}$.

    To finish the proof, we observe that, by maximality of the $\eps$-separated set, the balls $B_{W_1}(\nu_i,\eps)$, $i=1,\ldots,q$, must cover $\Gamma(I)$; otherwise there would exist $\nu\in\Gamma(I)$ such that $W_1(\nu,\nu_i)\geq\eps$ for every $i$, contradicting maximality. Hence, $N_{W_1}\bigl(\Gamma(I),\eps\bigr) \leq q \leq 1+\frac{L}{\eps}$ and, since $\Gamma(I)=\Gamma(\widehat I)$, theorem follows.
\end{proof}

%==================================================================
\section{An upper bound for the emergence of \texorpdfstring{$F$}{F}}\label{section:upper-bound-emergence}
%==================================================================

In this section, we prove Theorem~\ref{thm:emergencia-ordem-pol}, which implies Theorem~\ref{thm:low-emergence}. Before proving the theorem, we recall the following general criterion for low emergence, proved in~\cite[Proposition~2.6]{CostaSantiago2026}.

\begin{proposition}\label{prop:criterium-low-emerg}
Let $f:X \to X$ be a continuous map of a compact metric space, let $\mu$ be a Borel probability measure on $X$, and let $K\subset\Prob(X)$ be compact. Assume that, for $\mu$-almost every $x$, every accumulation point of $\bigl(\emp_n(x)\bigr)_{n\geq1}$ belongs to $K$. Then, for every $\eps>0$,
$$
\mathcal E_\mu(f)(\eps)
\leq
N_{W_1}(K,\eps).
$$
In particular, if $K$ has polynomial $W_1$-covering numbers, then $f$ has low emergence with respect to $\mu$.
\end{proposition}

With that proposition at hand, we need to understand the possible accumulation points of $\bigl(\emp_n(x,t)\bigr)_{n\geq1}$ for $\Leb$-a.e. $(x,t)\in B(m)\times \Sone$. This is the content of Theorem~\ref{thm.empirical}, which follows directly from the following proposition and Lemma~\ref{lem:empirical-segment}.

\begin{proposition}\label{prop:possible-accum}
    For Lebesgue-a.e. $(x,t)\in B(m)\times \Sone$ we have that either $\emp_n(x,t)\wto \mu \in \cG_{<0}$, or every accumulation measure of $\bigl(\emp_n(x,t)\bigr)_{n\geq1}$ has ergodic decomposition supported on $\cG_{\geq 0}$. This last condition means that, if $\eta$ is such an accumulation measure, then $\eta=\int \eta_e\,d\widetilde\nu(e)$ with $\widetilde\nu$-a.e. $\eta_e\in \cG_{\geq 0}$. 
\end{proposition}
\begin{proof}
By Remarks~\ref{rmk:lambda-ph}, $\Lambda\times \Sone$ is a partially hyperbolic attracting set for $F$ and $B(m)\times \Sone$ is a subset of its basin. By Remark~\ref{rmk:acc-u-gibbs}, there is a full Lebesgue measure subset of $B(m)\times \Sone$ such that for every point in this set, every accumulation point of its empirical measures is a $u$-Gibbs measure satisfying that almost every one of its ergodic components is also a $u$-Gibbs measure.

Observe that the measures in $\cG_{<0}$ are hyperbolic. Since they are also SRB, they are physical. Thus, $\cG_{<0}$ contains at most countably many measures.

Now, given $\nu \in \cG_{<0}$, define the set $E_\nu:=\{z\in B(m)\times \Sone \mid \supp(\nu)\subset \omega(z) \text{  and }z\notin B(\nu)\}$. Observe that if $z\in B(m)\times \Sone$ then $\omega(z)\subset \Lambda\times \Sone$, so it admits the same partially hyperbolic splitting as $\Lambda \times \Sone$. By \cite[Proposition 6.3]{CYZ2020}, $\Leb(E_{\nu})=0$, and since $\cG_{<0}$ is at most countable, the union $E_{-}:= \cup_{\nu \in \cG_{<0}}E_{\nu}$ still has zero Lebesgue measure.

Fix $z=(x,t)\in B(m)\times \Sone$ outside $E_{-}$ and inside the full Lebesgue measure given by Remark~\ref{rmk:acc-u-gibbs}. Let $\eta \in \acc(\emp_n(z))$ and write its ergodic decomposition as $\eta = \int \eta_e \, d\tilde{\nu}(e)$, with $\tilde{\nu}$-a.e. $e$ an ergodic $u$-Gibbs measure. By Lemma~\ref{lem:base-marginal}, $\pi_\ast \eta =m$, so $m=\int \pi_\ast \eta_e \, d\tilde{\nu}(e)$ and by ergodicity of $m$, $m=\pi_\ast \eta_e$ for $\tilde{\nu}$-a.e. $e$.

First, let us assume that there exists $\eta \in \acc(\emp_n(z))$ which is not supported in $\cG_{\geq 0}$, i.e., such that $\tilde{\nu}(\{e\mid \lambda^c(\eta_e)<0\})>0$. Thus, there exists an ergodic $u$-Gibbs measure $\eta_e$ such that $\lambda^c(\eta_e)<0$ and $\pi_\ast \eta_e=m$. Hence, $\eta_e \in \cG_{<0}$. Moreover, we have $\supp(\eta_e)\subset \supp(\eta)$ for $\tilde{\nu}$-a.e. $e$, so we can assume that the chosen ergodic $u$-Gibbs measure with positive center exponent satisfies this too. Since $\eta \in \acc(\emp_n(x,t))$, $\supp(\eta)\subset \omega(z)$. Since we chose $z\notin E_{-}$, since $\eta_e \in \cG_{<0}$ and since $\supp(\eta_e)\subset \omega(z)$ we must have, by the definition of $E_{-}$, that $z \in B(\eta_e)$. This proves $\emp_n (z)\wto \eta_e$.

Suppose now that the first alternative does not occur. In particular, there is no $\eta \in \acc(\emp_n(z))$ such that its ergodic decomposition satisfies $\tilde{\nu}(\{e\mid \lambda^c(\eta_e)<0\})>0$. Therefore, for every $\eta \in \acc(\emp_n(z))$ we have $\lambda^c(\eta_e)\geq 0$ for $\tilde{\nu}$-a.e. $e$, then since $\tilde{\nu}$-a.e. $e$ is an ergodic $u$-Gibbs measure with $\pi_\ast \eta_e=m$, we have $\eta_e\in \cG_{\geq 0}$ for $\tilde{\nu}$-a.e. $e$. This finishes the proof.
\end{proof}

\begin{remark}\label{rmk:caso-leq-1}
    Suppose now that $\#\cG_{\geq0}\leq1$ and set $K:=\overline{\cG_{<0}}\cup\cG_{\geq0}$. Proposition~\ref{prop:possible-accum} implies that, for Lebesgue-a.e. $(x,t)\in B(m)\times\Sone$, either $\emp_n(x,t)\wto\mu\in\cG_{<0}$, or every accumulation measure has ergodic decomposition supported on $\cG_{\geq0}$. Since $\#\cG_{\geq0}\leq1$, in the latter case every accumulation measure belongs to $\cG_{\geq0}$. Hence $\acc\bigl(\emp_n(x,t)\bigr)\subset K$ for Lebesgue-a.e. $(x,t)\in B(m)\times\Sone$. Since Corollary~\ref{cor:linear-covering} implies $N_{W_1}\bigl(\overline{\cG_{<0}},\eps\bigr)=O(\eps^{-1})$, and since $\cG_{\geq0}$ has at most one element, we have that $N_{W_1}(K,\eps)=O(\eps^{-1})$.
    
    Thus, to prove Theorem~\ref{thm:emergencia-ordem-pol}, we only need to deal with the case $\#\cG_{\geq0}\geq2$.
\end{remark}

\begin{lemma}\label{lemma:accumulacao no compacto}
    Assume that $\#\cG_{\geq0}\geq2$. There exists a compact set $K\subset \Prob(N\times\Sone)$ such that $N_{W_1}(K,\eps)=O(\eps^{-1})$ and, for Lebesgue-a.e. $(x,t)\in B(m)\times \Sone$, $\acc\bigl(\emp_n(x,t)\bigr)_{n\geq1}\subset K$.
\end{lemma}
\begin{proof}
Take $(x,t)\in B(m)\times\Sone$ in the full Lebesgue measure set given by Proposition~\ref{prop:possible-accum}. We distinguish the two alternatives of that proposition.

Suppose first that $\emp_n(x,t)\wto\mu$ for some $\mu\in\cG_{<0}$. Then $\acc\bigl(\emp_n(x,t)\bigr)=\{\mu\}\subset\cG_{<0}$. If, however, every accumulation measure of
$\bigl(\emp_n(x,t)\bigr)_{n\geq1}$ has ergodic decomposition
supported on $\cG_{\geq0}$, then by
Lemma~\ref{lem:empirical-segment} there exist $\mu^-,\mu^+\in\cG_{\geq0}$ such that $\acc\bigl(\emp_n(x,t)\bigr) \subset [\mu^-,\mu^+]$.

If $\mu^-=\mu^+$, then $\acc\bigl(\emp_n(x,t)\bigr)\subset\{\mu^-\}\subset\cG_{\geq0}\subset\Gamma(\widehat I)$. Otherwise, by the construction of $\mu^-$ and $\mu^+$ in the proof of Lemma~\ref{lem:empirical-segment}, they correspond to the endpoints of a connected component $J$ of $\widehat I\setminus S$. By the definition of $\Gamma$ on $J$, we have that $[\mu^-,\mu^+]=\Gamma(\overline J)\subset\Gamma(\widehat I)$. Then, in this case, $\acc\bigl(\emp_n(x,t)\bigr)\subset \Gamma(\widehat I)$. 

Let $K_-:=\overline{\cG_{<0}}$. Since $N\times\Sone$ is compact, $\Prob(N\times\Sone)$ is compact for the $W_1$ topology, and so $K_-$ is itself compact. Moreover, since $\cG_{<0}\subset \cM_{m}^{\mathrm{erg}}(F)$ there exists, by Corollary~\ref{cor:linear-covering}, $L_->0$ such that $N_{W_1}(\cG_{<0},\eps) \leq 1+\frac{L_-}{\eps}$. Since an $\eps/2$-dense set of $\cG_{<0}$ gives an $\eps$-dense set of $\overline{\cG_{<0}}$, 
$$
N_{W_1}(K_-,\eps)
\leq
N_{W_1}\left(\cG_{<0},\frac{\eps}{2}\right)
\leq
1+\frac{2L_-}{\eps},
$$
so $N_{W_1}(K_-,\eps)=O(\eps^{-1})$.

Let $K_+:=\Gamma(\widehat I)$. By continuity of $\Gamma$, the set $K_+$ is compact, and by
Theorem~\ref{thm:Gamma-covering} there exists $L_+>0$ such that
$$
N_{W_1}(K_+,\eps)
\leq
1+\frac{L_+}{\eps}.
$$

Setting $K:=K_-\cup K_+$, we get a compact set $K$ such that for Lebesgue-almost every $(x,t)\in B(m)\times\Sone$ we have
$\acc\bigl(\emp_n(x,t)\bigr)_{n\geq1}\subset K$, and
$$
N_{W_1}(K,\eps)
\leq
N_{W_1}(K_-,\eps)+N_{W_1}(K_+,\eps)
=
O(\eps^{-1}).
$$
This concludes the proof.
\end{proof}

\begin{proof}[Proof of Theorem~\ref{thm:emergencia-ordem-pol}]
Recall that the basins of the finitely many ergodic SRB measures
$$
m_{f,1},\ldots,m_{f,\ell}
$$
cover a full Lebesgue measure subset of $N$.

For each $i=1,\ldots,\ell$, Remark~\ref{rmk:caso-leq-1} in the case
$\#\cG_{\geq0}\leq1$, together with
Lemma~\ref{lemma:accumulacao no compacto} in the case
$\#\cG_{\geq0}\geq2$, gives a compact set $K_i\subset\Prob(N\times\Sone)$ such that, for Lebesgue-a.e. $(x,t)\in B(m_{f,i})\times\Sone$, $\acc\bigl(\emp_n(x,t)\bigr)_{n\geq1}\subset K_i$, and with $N_{W_1}(K_i,\eps)=O(\eps^{-1})$.

Define the compact set $K:=\bigcup_{i=1}^{\ell}K_i$. We have
$$
N_{W_1}(K,\eps)
\leq
\sum_{i=1}^{\ell}N_{W_1}(K_i,\eps)
=
O(\eps^{-1}),
$$
and since the basins $B(m_{f,i})$ cover Lebesgue almost every
point of $N$, then for Lebesgue-a.e. $(x,t)\in N\times\Sone$, $\acc\bigl(\emp_n(x,t)\bigr)_{n\geq1}\subset K$. By Proposition~\ref{prop:criterium-low-emerg}, we conclude that $\mathcal E_{\Leb}(F)(\eps) \leq N_{W_1}(K,\eps)=O(\eps^{-1})$, and this proves the theorem.
\end{proof}

In particular, Theorem~\ref{thm:emergencia-ordem-pol} implies that $F$ has low emergence, proving Theorem~\ref{thm:low-emergence}.
\color{black}

\bibliographystyle{alpha}
\bibliography{skew_products}

\end{document}